\documentclass[bibay2,reqno,jsl]{asl}

\usepackage{bussproofs}
\usepackage[hidelinks]{hyperref}
\usepackage{tabularx}

\usepackage[utf8]{inputenc}

\let\citet\cite
\let\citealt\cite
\def\citep#1{[\citeauth{#1} \citeyear{#1}]}

\let\epsilon\varepsilon
\def\Real{\mathbb{R}}
\def\lif{\mathbin{\rightarrow}}
\def\liff{\mathbin{\leftrightarrow}}
\def\et{{\ensuremath{\epsilon\tau}}}
\def\eps{\ensuremath{\epsilon}}
\def\meps#1#2{\epsilon_{#1}\,#2}
\def\mtau#1#2{\tau_{#1}\,#2}
\def\deg#1{\mathrm{deg}(#1)}
\def\rk#1{\mathrm{rk}(#1)}
\def\Lo{\mathbf{L}}
\def\IL{\ensuremath{\mathbf{H}}}
\def\LC{\ensuremath{\mathbf{LC}}}
\def\LIN{\ensuremath{\mathit{Lin}}}
\def\KC{\ensuremath{\mathbf{KC}}}
\def\J{\ensuremath{J}}
\def\Kur{\ensuremath{K}}
\def\CD{\ensuremath{\mathit{CD}}}
\def\Eout{\ensuremath{Q_\exists}}
\def\Aout{\ensuremath{Q_\forall}}
\def\Wel{\ensuremath{\mathit{Wel}_1}}
\def\DWel{\ensuremath{\mathit{Wel}_2}}
\def\Bm#1{{\ensuremath{B_{#1}}}}
\def\CL{\ensuremath{\mathbf{C}}}
\def\G{\ensuremath{\mathbf{G}}}
\def\IP{\ensuremath{\mathit{IP}_<}}
\def\IR{\ensuremath{\mathit{IR}_<}}
\def\Ax{\mathit{Ax}}

\def\fCenter{\lif}

\def\pl#1{#1}
\def\ql#1{\mathbf{Q}#1}
\def\etl#1{#1\et}
\def\eo{{<\eps}}
\def\etx#1{#1\et}
\def\proves#1{#1 \vdash}

\let\mathtau\tau
\def\tau{\ensuremath{\mathtau}}

\theoremstyle{plain}
\newtheorem{thm}{Theorem}[section]
\newtheorem{lem}[thm]{Lemma}
\newtheorem{prop}[thm]{Proposition}

\newtheorem{cor}[thm]{Corollary}
\theoremstyle{definition}
\newtheorem{defn}[thm]{Definition}

\newtheorem{rem}[thm]{\em Remark}

\title{Epsilon Theorems in Intermediate Logics}


\dedicatory{Dedicated to the memory of Grisha Mints}

\author[Matthias Baaz]{Matthias Baaz$^*$}
\revauthor{Baaz, Matthias}
\thanks{$^*$Research supported by the Austrian Science Fund in projects P31063, I4427, and P31955.} 
\address{Institute of Discrete Mathematics and Geometry\\
Vienna University of Technology\\
Wiedner Hauptstrasse 8--10\\
1040 Vienna, Austria}
\email{baaz@logic.at}

\author[Richard Zach]{Richard Zach$^\dag$}
\revauthor{Zach, Richard}
\thanks{$^\dag$Research supported by the Natural Sciences and Engineering Council of Canada.}
\address{Department of Philosophy\\
University of Calgary\\
2500 University Dr NW\\
Calgary AB T2N~1N4, Canada}
\email{rzach@ucalgary.ca}
\urladdr{https://richardzach.org/}

\begin{document}

\begin{abstract}
  Any intermediate propositional logic (i.e., a logic including
  intuitionistic logic and contained in classical logic) can be
  extended to a calculus with epsilon- and tau-operators and critical
  formulas.  For classical logic, this results in Hilbert's
  \eps-calculus. The first and second \eps-theorems for classical
  logic establish conservativity of the \eps-calculus over its
  classical base logic. It is well known that the second \eps-theorem
  fails for the intuitionistic \eps-calculus, as prenexation is
  impossible. The paper investigates the effect of adding critical
  \eps- and \tau-formulas and using the translation of quantifiers
  into \eps- and \tau-terms to intermediate logics. It is shown that
  conservativity over the propositional base logic also holds for such
  intermediate \et-calculi. The ``extended'' first \eps-theorem holds
  if the base logic is finite-valued G\"odel-Dummett logic, fails
  otherwise, but holds for certain provable formulas in
  infinite-valued G\"odel logic. The second \eps-theorem also holds
  for finite-valued first-order G\"odel logics. The methods used to
  prove the extended first \eps-theorem for infinite-valued G\"odel
  logic suggest applications to theories of arithmetic.
\end{abstract}

\maketitle

\section{Introduction}

The \eps-calculus was originally introduced by Hilbert as a
formalization of classical first-order logic. It is a way to reduce
proofs in first-order logic to proofs in propositional logic from
so-called critical formulas, where the role of quantifiers is taken
over by certain terms.  The \eps-calculus was the basis for Hilbert's
approach to proof theory (in particular, consistency proofs). 
It still is a useful logical formalism with interesting properties and
theoretical and practical
applications. 

The \eps-calculus is formulated by allowing for terms of the form
$\meps{x}{A(x)}$ for any formula $A(x)$ with $x$ free.  A formula of
the form $A(t) \lif A(\meps{x}{A(x)})$ is called a \emph{critical
formula} belonging to $\meps{x}{A(x)}$. A proof in the \eps-calculus
is a proof in the quantifier-free fragment of classical logic from
critical formulas. It is now possible to define the existential
quantifier by $\exists x\, A(x) \equiv A(\meps{x}{A(x)})$, and---in
classical logic---the universal quantifier by $\forall x\, A(x) \equiv
A(\meps{x}{\lnot A(x)})$. A formula~$A$ with quantifiers can thus be
translated into a formula~$A^\eps$ with \eps-terms but without
quantifiers. A formula is provable in classical first-order
logic~$\ql{\CL}$ iff its translation is provable in the \eps-calculus.

Hilbert proved two fundamental results about the \eps-calculus:
\begin{enumerate}
  \item The extended first \eps-theorem: If $A(\vec e)$ is derivable
    in the \eps-calculus ($\vec e$ a tuple of $\eps$-terms), then
    there are tuples of terms $\vec t_1$, \dots, $\vec t_n$ such that
    $A(\vec t_1) \lor \ldots \lor A(\vec t_n)$ is provable in
    propositional logic. 
  \item The second \eps-theorem: If $A^\eps$ is the standard
  \eps-translation of a first-order formula derivable in the
  \eps-calculus, $A$~is derivable in first-order logic.
\end{enumerate}
The extended first \eps-theorem has two important consequences. The
first consequence is what Hilbert simply
called the first \eps-theorem: If $A$ is \eps-free and
derivable in the \eps-calculus, it is derivable in the quantifier-free
fragment of first-order logic, the so-called elementary calculus of
free variables (i.e., without critical formulas,
indeed, without any use of \eps-terms). This implies that the
\eps-calculus is conservative over propositional logic for
quantifier-free formulas without identity. The second consequence is
Herbrand's theorem for existential formulas: If $\exists \vec x\,
A(\vec x)$ is provable in the \eps-calculus, then some disjunction
$A(\vec t_1) \lor \ldots \lor A(\vec t_n)$ is provable in
propositional logic alone, i.e., is a tautology.\footnote{See
\cite{HilbertBernays1939} for the first presentation of the
\eps-theorems, \cite{AvigadZach2002} for a survey, and
\cite{MoserZach2006} for a modern presentation.} 

In contrast to other proof theoretic methods which also yield the
existence of Herbrand disjunctions (such as cut-elimination), the
proof based on the extended first \eps-theorem, and consequently the
length of the Herbrand disjunction the proof yields, is insensitive to
the propositional complexity of the original proof. This is an
important advantage of methods based on the \eps-calculus.

Hilbert's results and proofs make essential use of classical
principles, especially the law of excluded middle. The question
naturally arises whether the results can also be obtained for weaker
logics, and whether the same proof methods can be used, i.e., whether
the use of excluded middle can be avoided.  In this paper, we
investigate \eps-calculi for intermediate logics, i.e., logics between
intuitionistic and classical logic, and specifically the question of
when the extended first \eps-theorem holds in such logics. Well-known
examples of intermediate logics are Jankov's logic of weak excluded
middle and finite- and infinite-valued G\"odel-Dummett logics.

We consider only intermediate logics for two reasons. One is that Hilbert's
methods rely essentially on the deduction theorem, and this holds in
intermediate logics but not in many other logics.  The other is that
\eps-calculi for intuitionistic and intermediate logics are of
independent interest. What is the effect of adding \eps-operators to
intermediate logics? When does the extended first \eps-theorem hold?
When is the \eps-calculus for a logic conservative over the
propositional base logic? We also, for the most part, discuss only
pure logics, i.e., logics without identity.

In intermediate logics it is necessary to introduce a separate
\tau-operator which defines the universal quantifier.  Whereas
$A(\meps x A(x))$ translates $\exists x\,A(x)$, $A(\mtau x A(x))$
translates $\forall x\,A(x)$.  The corresponding critical formulas are
those of the form $A(\mtau x A(x)) \lif A(t)$.  Weakening the logic
makes the addition of \tau{} necessary, since the equivalence of
$\forall x\, A(x)$ and $A(\meps x \lnot A(x))$ relies on the schema of
contraposition; the equivalence of $\forall x\, A(x)$ and $A(\mtau x
A(x))$ does not. The system resulting from a propositional
intermediate logic $\pl{\Lo}$ by adding \eps- and \tau-terms and
critical formulas is called the \et-calculus for~$\pl{\Lo}$.

It is well known that adding the \eps-operator to intuitionistic logic
in a straightforward way results in translations~$A^\eps$ of
intuitionistically invalid formulas~$A$ becoming provable.
\cite{Mints1977,Mints1990} has investigated different systems based on
intuitionistic logic with \eps-operators which are conservative and in
which the \eps-theorem holds. He allows the use of \eps-terms only
when $\exists x\,A(x)$ has been established; other approaches (e.g.,
\citealt{Shirai1971}) use existence predicates to accomplish the same.
We investigate the basic \et-calculus without this assumption. In our
\et-calculi, \et-terms are treated syntactically as something like
Skolem functions rather than semantically as choice operators.

We begin (Section~\ref{prelims}) by introducing intermediate logics.
In Section~\ref{sec:embedding} we introduce the \et-calculus and
consider the differences in formulas provable in a logic vs.\ those
provable in the corresponding \et-calculus.  This is essentially the
question of whether the addition of \eps- and \tau-terms and critical
formulas allows the derivation (or requires the validity) of formulas
not provable in the base logic.

We investigate in detail \emph{which} intuitionistically invalid
formulas are provable in \et-calculi for intermediate logics in Section~\ref{sec:quant-shift}.
We show that quantifier shift schemas play a special role here.  All
but three of these are valid in intuitionistic logic. The addition of
\et-terms and critical formulas results in the provability of the
remaining three. Consequently, no intermediate first-order logic in
which one of these three quantifier shifts is unprovable can have the
second \eps-theorem (Proposition~\ref{prop:qsprovet}). This includes
intuitionistic logic itself, logics complete for non-constant domain
Kripke frames, and infinite-valued G\"odel-Dummett logic.

In Section~\ref{sec:conservative}, we show that conservativity for the
propositional fragment nevertheless holds for all intermediate logics
(Theorem~\ref{thm:conservative}). This in itself is a surprising
result, even though the proof is very easy. It holds for theorems of
the logics with or without identity, however not in general for
provability from theories.

In Sections~\ref{sec:negative}--\ref{sec:hb-elim} we give a complete
characterization of the intermediate logics where the extended first
\eps-theorem holds. We show (Theorem~\ref{thm:forking}) that whenever
it holds, the underlying logic must prove (or validate) a
sentence~$\Bm m$ of the form
\[
(A_1 \lif A_2) \lor (A_2 \lif A_3) \lor \dots \lor (A_m \lif A_{m+1}).
\]
No $B_m$ is provable in intuitionistic logic, any logic complete for
Kripke frames with branching worlds, or in infinite-valued
G\"odel-Dummett logic. Consequently, the extended first \eps-theorem
does not hold for these logics.

Provability of $\Bm m$ is also a sufficient condition: We show that
the extended first \eps-theorem holds whenever the underlying logic
proves at least one~$\Bm m$.  The argument follows the idea of
Hilbert's proof, but does not make use of excluded middle. In order to
establish the result, we provide a more fine-grained analysis of the
proof of the extended first \eps-theorem. We first introduce the
notion of an elimination set in Section~\ref{sec:elim-sets}: a set of
terms which can replace an \eps-term in a proof and render the
corresponding critical formulas redundant, using only the resources of
the underlying propositional logic. This is the part of Hilbert's
proof that uses excluded middle. We isolate here how excluded middle
is used. In Section~\ref{sec:positive}, we show that logics that prove
some $\Bm m$ also have elimination sets.

If such elimination sets exist, the procedure given by Hilbert and
Bernays can be used for a proof of the extended first \eps-theorem
(Section~\ref{sec:hb-elim}). This establishes the second half of the
characterization, that the extended first \et-theorem holds in logics
that prove $\Bm m$, i.e., the finite-valued G\"odel logics
(Theorem~\ref{thm:et-Gm}). By putting emphasis on elimination sets, we
can also show that in logics in which the extended first \eps-theorem
does not hold in general, it may still hold for formulas of a special
form. This allows us to show that the extended first \et-theorem holds
for negated formulas in Jankov's logic of weak excluded middle
(Theorem~\ref{et-neg-KC}).  We also show that the first \et-theorem
(for theorems not containing \et-terms) holds for infinite-valued
G\"odel-Dummett logic (Theorem~\ref{thm:weak-eps-LC}).

The extended first \et-theorem is closely related to Herbrand's
theorem. We discuss this connection, as well as the second
\et-theorem, in Section~\ref{sec:herbrand}. We show that the second
\et-theorem holds for any intermediate predicate logic that proves all
quantifier shifts and has the extended first \et-theorem
(Proposition~\ref{prop:second}), e.g., $\ql{\LC_m} + \CD$, first-order
finite-valued G\"odel logic. 

In the case of some proofs, it is possible to eliminate \et-terms in a
simplified way where the cases that require the presence of $\Bm m$ do
not arise, and the linearity schema $\LIN$ ($(A \lif B) \lor (B \lif A)$) is enough. Although we cannot give an
independent characterization of the proofs or theorems for which this
is the case, the simplified procedure will sometimes terminate and
produce a Herbrand disjunction. Conversely, if a Herbrand disjunction
exists, there is always a proof of the original formula for which the
procedure terminates and produces the Herbrand disjunction
(Section~\ref{sec:LIN}).

This result sheds light on the conditions under which (a version of)
Hilbert's method which uses principles weaker than excluded middle
produces a Herbrand disjunction.  In fact, a similar method can be
used to give a partial \eps-elimination procedure for number theory,
where linearity of the natural order of~$\mathbb{N}$ plays a similar
role as the schema of linearity does in the case of logic
(Section~\ref{sec:order}).

\begin{table}
  \begin{tabular}{l|llr}
    \hline\hline
    & Propositional logics / Axioms\\
    \hline
    \IL{} & Intuitionistic logic & \\
    \KC & Logic of weak excluded middle: $\IL + J = \lnot A \lor \lnot\lnot
    A$\\
    \LC{} & infinite-valued G\"odel logic, linear Kripke frames\\
    & $\IL + \LIN = (A \lif B) \lor (B \lif A)$\\
    $\LC_m$ & $m$-valued G\"odel logic, 
     linear Kripke frames of length $< m$\\
     & $\IL + B_m = (A_1 \lif A_2) \lor \dots \lor 
     (A_m \lif A_{m+1})$\\
    \CL & Classical logic: 
      $\IL + A \lor \lnot A$, $\IL + B_2$\\
    
    \hline
    & First-order logics / Axioms\\
    \hline
    $\ql\IL$ & Intuitionistic logic \\
    $\ql\KC$ & Weak excluded middle: $\ql\IL + J$\\
    $\ql\LC$ & Linear Kripke frames: $\ql\IL + \LIN$\\
    $\ql\LC_m$ & $\ql{\IL} + B_m$\\
    $\G_\Real$ & G\"odel logic on $[0,1]$, constant-domain linear Kripke frames\\
          & $\ql\LC + \CD = \forall x(A(x) \lor B) \lif (\forall
          x\,A(x) \lor B)$\\
    $\G_0$ & G\"odel logic on $\{0\} \cup [1/2,1]$\\
     & $\ql\LC + \CD + K = \forall x\lnot\lnot A(x) \lif \lnot\lnot
     \forall x\,A(x)$\\
    $\G_m$ & $m$-valued G\"odel logic: $\ql{\IL} + B_m + \CD$\\
    $\ql\CL$ & Classical logic\\
  \hline\hline
  \end{tabular}
  \vspace{2pt}
  \caption{Intermediate logics considered}
  \label{logicsstable}
  \end{table}

\begin{table}
\begin{tabular}{@{}l|l|l|l}
  \hline\hline
  Result &  & Base logic & \\
  \hline
  Conservativity over & Yes &
  Any & 
  \ref{thm:conservative}\\
  propositional logic & & &\\
  First \et-Theorem & Yes & \LC & \ref{thm:weak-eps-LC}\\
  & Yes & Any (for negated formulas) & \ref{neg-IL-first-eps}\\
  Extended First \eps-Theorem & No & Any except $\LC_m$ &
  \ref{cor:negative}\\
  & Yes & \CL & \ref{first-eps-cl} \\
  & Yes & \LC, \KC{} (for negated formulas) &
  \ref{et-neg-KC}\\
  & Yes & $\LC_m$ & \ref{thm:et-Gm}\\
  Second \et-Theorem & No & $\G_\Real$, $\ql{\LC_m}$ & \ref{prop:no-second}\\
  & Yes & $\G_m$, $\ql\CL$ & \ref{second-gm}\\
\hline\hline
\end{tabular}
\vspace{2pt}

\caption{Epsilon theorems for intermediate logics}
\label{resultstable}
\end{table}

\section{Preliminaries}\label{prelims}

Intermediate propositional logics have been investigated extensively
since the 1950s. Intermediate predicate logics are comparatively less
well understood; however, they constitute an active area of research
(see \citealt{GabbayShehtmanSkvortsov2009}).  We begin by collecting
some preliminary definitions.

All the logics we consider are formulated in the standard language of
intuitionistic and classical logic with propositional connectives
$\land$, $\lor$, $\lif$, the constant~$\bot$ for absurdity, and the
quantifiers $\forall$ and~$\exists$. $\lnot A$ is defined as $A \lif
\bot$ and $\top$ as $\lnot \bot$. Terms and atomic formulas are
defined as usual. We allow $0$-place predicate symbols, i.e.,
propositional variables. A formula containing no quantifiers is called
\emph{quantifier-free} and a quantifier-free formula with only
$0$-place predicate symbols is called \emph{propositional}. For the
most part we will consider pure logics, i.e., logics not involving the
identity predicate~$=$. We assume that function symbols are available,
however.

We use $A$, $B$, \dots, as metavariables for formulas, and $t$, $s$,
\dots, as metavariables for terms. We write $A(x)$ to indicate that
$x$ occurs free in~$A$. The result of substituting $s$ for all (free)
occurrences of~$x$ in a term~$t$ or a formula~$A$ is indicated by
$t[s/x]$ or $A[s/x]$. When it is clear which variable~$x$ is intended,
we write $A(s)$ for $A[s/x]$. The result of replacing every occurrence
of a term $t$ by a term~$s$ in a formula~$A$ is indicated by $A[s/t]$.
A \emph{substitution instance} of a formula~$A$ is any formula
resulting from $A$ by uniformly replacing any number of atomic
formulas $P(t_1, \dots, t_n)$ by formulas $B[t_1/x_1, \dots, t_n/x_n]$
(in such a way that free variables of $B$ are not captured by
quantifiers in~$A$ of course). In particular, a substitution instance
of a propositional formula~$A$ is any formula resulting from $A$ by
uniformly replacing any number of $0$-place predicate symbols~$P$ by
formulas~$B$. 

\begin{defn}
  A \emph{propositional logic}~$\pl{\Lo}$ is a set of propositional
  formulas closed under substitution and modus ponens, i.e., if $\Lo$
  contains $A \lif B$ and $B$ it also contains~$B$. A \emph{predicate
  logic} is a set of formulas closed under subsitution, modus ponens,
  and the quantifier rules
  \[
  \Axiom$B \fCenter A(x)$
  \UnaryInf$B \fCenter \forall y\, A(y)$
  \DisplayProof
  \quad
  \Axiom$A(x) \fCenter B$
  \UnaryInf$\exists y\, A(y) \fCenter B$
  \DisplayProof
  \]
  which are subject to the eigenvariable condition: $x$ must not be free
  in the conclusion.
\end{defn}

Logics can be characterized by the formulas derivable from a set of
axiom schemas by modus ponens.  This is of course equivalent to
closing the axioms under substitution and modus ponens. For instance,
intuitionistic propositional logic~$\pl{\IL}$ is obtained from the
propositional axioms given in \citet[4.1]{Troelstra1988}. Classical
propositional logic $\pl{\CL}$ and predicate logic $\ql{\CL}$ are
obtained by adding~$\lnot A \lor A$ (or alternatively $\lnot\lnot A
\lif A$) as an axiom.

\begin{defn}
  An \emph{intermediate propositional logic}~$\pl{\Lo}$ is a
  propositional logic that contains $\pl{\IL}$ and is contained
  in~$\pl{\CL}$.
\end{defn}

The following intermediate propositional logics will play important
roles:
\begin{enumerate}
  \item \LC, characterized alternatively as the formulas valid on
  linearly ordered Kripke frames or as infinite-valued Gödel
  logic,\footnote{See \citep{Dummett1959}. Note that although there is
  only one infininite-valued G\"odel logic considered as a set of
  tautologies, there are infinitely many different consequence
  relations on infinite truth-value sets with the G\"odel truth
  functions \citep{BaazZach1998a}.} axiomatized over \IL{} using the
  schema
  \begin{equation*}
    (A \lif B) \lor (B \lif A). \tag{\LIN}
  \end{equation*}
  \item $\LC_m = \LC + \Bm m$, characterized as formulas valid on
  linearly ordered Kripke frames of height $< m$, or as the G\"odel
  logic on~$m$ truth values, also known as $\textbf{S}_{m-1}$
  \citep{Hosoi1966}. Here, $\Bm m$ is:
  \begin{equation*}
    (A_1 \lif A_2) \lor (A_2 \lif A_3) \lor \dots \lor 
    (A_m \lif A_{m+1}). \tag{\Bm m}
  \end{equation*}
  \item \KC, the logic of weak excluded middle \citep{Jankov1968},
  axiomatized over \IL{} using the schema
  \begin{equation*}
    \lnot A \lor \lnot\lnot A. \tag{\J}
  \end{equation*}
\end{enumerate}

If $\pl{\Lo}$ is an intermediate propositional logic, we can consider
the corresponding ``elementary calculus,'' i.e., the system obtained
by replacing propositional variables with atomic formulas of a
first-order language, with or without identity. We will be interested
in formulas provable in such an elementary calculus from a set of
assumptions~$\Gamma$.

\begin{defn}\label{defn:elem-calc-proof} Suppose $\Lo$ is an
  intermediate logic. A proof~$\pi$ of~$A$ from~$\Gamma$ in the
  corresponding \emph{elementary calculus} is a sequence $A_1$, \dots,
  $A_n =A$ of quantifier-free formulas such that each~$A_i$ is either a
  substitution instance of a formula in~$\Lo$, is in~$\Gamma$, or
  follows from formulas $A_k$ and $A_l$ ($k, l < i$) by modus ponens,
  and $A_n \equiv A$. We then write $\Gamma \vdash_{\Lo}^\pi A$. If
  such a proof~$\pi$ exists we write $\Gamma \vdash_{\Lo}$ and if
  $\Gamma$ is empty, simply $\proves{\Lo} A$. If $\Lo$ is not mentioned,
  we mean~\IL.
\end{defn}

Given a propositional logic~$\pl{\Lo}$, and possibly a set of
additional axiom schemas~$\Ax$ involving quantifiers, we can generate a
predicate logic:

\begin{defn}
If $\pl{\Lo}$ is a propositional logic, and $\Ax$ is a set of
formulas, then $\ql{\Lo} + \Ax$ is the smallest predicate logic
containing~$\pl{\Lo}$, the standard quantifier axioms
\begin{align*}
\forall x\, A(x) & \lif A(t) \text{\quad and}\\
A(t) & \lif  \exists x\, A(x),
\end{align*}
and the formulas in~$\Ax$.
\end{defn}

If $\ql{\Lo} + \Ax$ contains intuitionistic predicate logic~$\ql{\IL}$
and is contained in classical predicate logic~$\ql{\CL}$, it is called
an \emph{intermediate predicate logic}.

\begin{defn}\label{defn:proof} Suppose $\Lo$ is a propositional logic.
  A proof~$\pi$ of $A$ in $\ql{\Lo} + \Ax$ is a sequence $A_1$, \dots,
  $A_n = A$ of formulas of predicate logic such that each $A_i$ is either
  a substitution instance of a formula in~$\Lo$, of a standard
  quantifier axiom, of a formula in~$\Ax$, or follows from previous
  formulas by modus ponens or a quantifier rule. We then write $
  \ql{\Lo} + \Ax \vdash^\pi A$. If such a proof~$\pi$ exists we write
  $\ql{\Lo} + \Ax \vdash A$.
\end{defn}

$A$ is a formula in $\ql{\Lo} + \Ax$ iff $\proves{\ql{\Lo}
+ \Ax} A$.  If $\Lo{}$ itself is characterized by a set of
propositional axioms, it is enough to require substitution instances
of axioms of~$\Lo{}$ in the above definition. For instance,
taking~\IL{} as above, the proofs in intuitionistic predicate
logic~$\ql{\IL}$ are just the proofs in the system $H_2$-\textbf{IQC}
of \citet[4.3]{Troelstra1988}.

If $\Ax$ is empty, $\ql{\Lo} = \ql{\Lo} + \Ax$ is the weakest pure
intermediate predicate logic extending~$\pl{\Lo}$. For instance,
$\ql{\LC}$ is the weakest predicate logic obtained from~$\LC$, and is
axiomatized by $\ql{\IL} + \LIN$. It is complete for linearly-ordered
Kripke frames (see \cite{Corsi1992,Skvortsov2005}). 

It is possible to extend the weakest predicate logic of an
intermediate propositional logic~$\pl{\Lo}$ by adding
intuitionistically invalid schemas~$\Ax$ involving quantifiers. Some
important examples are the constant domain principle,
\begin{align*}
  \forall x(A(x) \lor B) & \lif (\forall x\,A(x) \lor B), \tag{\CD}
\intertext{the double negation shift (or Kuroda's principle),}
  \forall x\,\lnot\lnot A(x) & \lif \lnot\lnot \forall x\,A(x) \tag{\Kur}
\intertext{and the quantifier shifts}
  (B \lif \exists x\,A(x)) & \lif \exists x(B \lif A(x)) \tag{\Eout}\\
  (\forall x\,A(x) \lif B) & \lif \exists x(B \lif A(x)) \tag{\Aout}
\end{align*}
$\ql{\LC} + \CD$ axiomatizes the formulas valid in linearly-ordered
Kripke frames with constant domains, and also the first-order G\"odel
logic~$\G_\Real$ of formulas valid on the interval~$[0, 1]$. $\ql{\IL}
+ \Kur$ characterizes the formulas valid on Kripke frames with the
McKinsey property and $\ql{\LC} + \CD + \Kur$ is the logic of linear
Kripke frames with maximal element. It is also the first-order G\"odel
logic~$\G_0$ of formulas valid on $\{0\} \cup [1/2,1]$. $\ql{\LC_m}$
is not complete for linear Kripke frames of height $< m$ (contrary to
what one might expect). However, $\G_m = \ql{\LC_m} + \CD$ is complete
for linearly ordered Kripke frames of height $<m$ with constant
domains. It is also the $m$-valued first-order G\"odel
logic~$\G_m$.\footnote{See
\citet{Skvortsov2005} and \citet{GabbayShehtmanSkvortsov2009} for the mentioned
Kripke completeness results and \citet{BaazPreiningZach2007} for the
characterizations in terms of G\"odel truth value sets. ($\G_m$,
$\G_\Real$,and $\G_0$ are the only axiomatizable first-order G\"odel
logics.)}

\section{$\et$-Calculi for Intermediate Logics}
\label{sec:embedding}

Formulas and terms of the \et-calculus are defined by simultaneous
induction, allowing that if $A(x)$ is a formula already defined, then
$\meps x{A(x)}$ and $\mtau x{A(x)}$ are terms.  In $\meps x{A(x)}$ and
$\mtau x{A(x)}$ the variable~$x$ is bound. We call terms of the form
$\meps x{A(x)}$ \emph{\eps-terms} and those of the form $\mtau
x{A(x)}$, \emph{\tau-terms} (collectively: \emph{\et-terms}).

As usual, we consider \eps- and \tau-terms to be identical up to
renaming of bound variables, and define substitution of \et-terms into
formulas, as in $A(\meps{x}{A(x)})$, so that bound variables are
tacitly renamed to avoid clashes.

\begin{defn}
  A \emph{critical formula belonging to~$\meps{x}{A(x)}$}
  is any formula of the form $A(t) \lif A(\meps{x}{A(x)})$.
  
  A critical formula belonging to~$\mtau{x}{A(x)}$ is any formula
  of the form $A(\mtau{x}{A(x)}) \lif A(t)$.
\end{defn}

\begin{defn}\label{defn:etcalc} Suppose $\pl{\Lo}$ is an intermediate
propositional logic. An \emph{\et-proof} $\pi$ of~$B$ is a proof in
the elementary calculus of~$\Lo$ from critical formulas~$\Gamma$ of
the form
\begin{align*}
  A(t) & \lif A(\meps x A(x))\\
  A(\mtau x A(x)) & \lif A(t).
\end{align*}
We write $\proves{\etl{\pl{\Lo}}} B$ if such a $\pi$ exists, or
$\Gamma \vdash_{\etl{\Lo}}^\pi B$ when we want to identify the
critical formulas and the proof~$\pi$.

The \emph{pure \et-calculus}~$\etl{\Lo}$ of $\Lo$ is the set of
quantifier-free formulas that have \et-proofs. 
\end{defn}

We are interested in the relationships between intermediate predicate
logics $\ql{\Lo} + \Ax$ and the \et-calculus~$\etl{\Lo}$ of their
propositional fragment~$\pl{\Lo}$. Since the language of
$\etl{\pl{\Lo}}$ does not contain quantifiers, we must define a
translation of predicate formulas that do contain them into the
language of the \et-calculus.

\begin{defn}\label{defn:et-trans} The \emph{$\et$-translation}~$A^\et$
  of a formula~$A$ is defined as follows:
  \begin{align*}
    A^\et & = A \text{\ if $A$ is atomic}\\
    (A \land B)^\et & = A^\et \land B^\et & 
    (A \lor B)^\et & = A^\et \lor B^\et \\
    (A \lif B)^\et & = A^\et \lif B^\et & 
    (\lnot A)^\et & = \lnot A^\et \\
    (\exists x\, A(x))^\et & = A^\et(\meps x A(x)^\et) & 
    (\forall x\, A(x))^\et & = A^\et(\mtau x A(x)^\et)
  \end{align*}
\end{defn}

Again, substitution of \et-terms for variables must be understood
modulo renaming of bound variables so as to avoid clashes. Clearly, if
$A$ contains no quantifiers, then $A^\et = A$.

The point of the classical \eps-calculus is that it can
\emph{replace} quantifiers and quantifier inferences. And indeed, in
classical first-order logic, a first-order formula~$A$ is provable
iff its translation~$A^\eps$ is provable in the pure
\eps-calculus. The ``if'' direction is the content of the second
\eps-theorem, while the ``only if'' direction follows more simply by
translating derivations.

We defined the \et-calculus on the basis of a propositional
logic~$\pl{\Lo}$. It is also possible to define an ``extended''
\et-calculus by adding \et-terms and critical formulas to the full
first-order language including quantifiers, and then considering
proofs in~$\ql{\Lo} + \Ax$ from critical formulas.

\begin{defn}\label{defn:etxcalc}
A proof $\pi$ in $\etx{(\ql{\Lo}+\Ax)}$ of~$A$ is a sequence of formulas
of the \et-calculus (possibly containing quantifiers) ending in~$A$ in
which each formula is either an instance of a formula in~$\Lo$, an
instance of a schema in~$\Ax$, a standard quantifier axiom, a
critical formula, or follows from preceding formulas by modus ponens
or a quantifier rule. We write $\proves{\etx{(\ql{\Lo} + \Ax)}} A$ if
such a $\pi$ exists.

The \emph{extended \et-calculus}~$\etl{(\ql\Lo + \Ax)}$ of $\ql\Lo +
\Ax$ is the set of formulas that have proofs in $\etx{(\ql{\Lo} +
\Ax)}$. 
\end{defn}

If $A$ is quantifier-free and $\proves{\etl{\pl{\Lo}}} A$
then $\proves{\etx{(\ql{\Lo} + \Ax)}} A$. One may wonder, however, if
$\etx{(\ql{\Lo} + \Ax)}$ is stronger than $\etl{\pl{\Lo}}$ in the
sense that for some formulas~$A$, $\proves{\etx{(\ql{\Lo} + \Ax)}} A$
but not $\proves{\etl{\pl{\Lo}}} A^\et$. This is not so as long as the
\et-translations of the axioms~$\Ax$ of $\ql{\Lo} + \Ax$ are provable
in $\etl{\pl{\Lo}}$; then the extended \eps-calculus is conservative
over the pure \eps-calculus.

\begin{lem}\label{embed} Suppose $\ql{\Lo} + \Ax$ is an intermediate
  predicate logic, and for each quantifier axiom $B \in \Ax$,
  $\proves{\etl{\pl{\Lo}}} B^\et$.  If $\proves{\etx{(\ql{\Lo} +
  \Ax)}} A$, then $\proves{\etl{\pl{\Lo}}} A^\et$.
\end{lem}

\begin{proof} By standard proof transformations we may assume that the
proof~$\pi$ in $\etx{(\ql{\Lo} + \Ax)}$ is such that every formula is
used as a premise of at most one modus ponens or quantifier inference,
and that the eigenvariables of all quantifier inferences are distinct
(the proof is \emph{regular}). The proof then proceeds by induction on
the length of~$\pi$.

Any formula $B$ in~$\pi$ that is not the conclusion of an inference is
either in~$\Lo$, a critical formula, in~$\Ax$, or a standard
quantifier axiom. Then $B^\et$ is also either an axiom in $\pl{\Lo}$,
a critical formula (as can easily be seen from the definition of the
\et-translation), or, if $B \in \Ax$ then
$\proves{\etl{\pl{\Lo}}}B^\et$ (by hypothesis). If $B$ is a standard
quantifier axiom, its \et-translation is a critical formula:
\begin{align*}
  [A(t) \lif \exists x\,A(x)]^\et &= A^\et(t) \lif A^\et(\meps x A^\et(x))\\
  [\forall x\,A(x) \lif A(t)]^\et &= A^\et(\mtau x A^\et(x)) \lif A^\et(t)
\end{align*}

If $B \lif A(x)$ is derivable (where the eigenvariable~$x$ is not free
in $B$), then so is $B \lif A(\mtau{x}{A(x)})$, by substituting
$\mtau{x}{A(x)}$ everywhere $x$ appears free in the part of~$\pi$
leading to $B \lif A(x)$, and renaming bound variables to avoid
clashes. Thus, if $B$ is the conclusion of a quantifier inference,
$\proves{\etl{\pl{\Lo}}} B^\et$. Similarly, if $B(x) \lif A$ is
derivable, so is $B(\meps{x}{B(x)}) \lif A$.
(Cf.~\cite[Lemma~7]{MoserZach2006}.)
\end{proof}

As we'll see in Section~\ref{sec:quant-shift}, the quantifier axioms
of the intermediate predicate logics considered in the preceding
section satisfy the condition that the \et-translations of their
additional quantifier axioms are derivable from critical formulas
alone (i.e., already in~$\etl{\Lo}$).

\begin{cor}
  If $\ql{\Lo} + \Ax$ satisfies the conditions of Lemma~\ref{embed},
  and $A$ contains no quantifiers, then $\proves{\etx{(\ql{\Lo}+\Ax)}}
  A$ only if $\proves{\etl{\pl{\Lo}}} A$.
\end{cor}

\begin{proof}
  If $A$ contains no quantifiers, then $A^\et \equiv A$.
\end{proof}

\begin{prop}
  For any intermediate predicate logic $\ql{\Lo} + \Ax$,
  \begin{align*}
  & \proves{\etx{(\ql{\Lo} + \Ax)}} \forall x\, A(x) \liff A(\mtau x{A(x)}) \text{ and}\\
  &\proves{\etx{(\ql{\Lo}+\Ax)}} \exists x\, A(x) \liff A(\meps x{A(x)}).
  \end{align*}
\end{prop}

\begin{proof}
  In each case, one direction is an instance of the corresponding
  quantifier axiom, and the other direction follows from a critical
  formula by the corresponding quantifier rule.  For instance,
  $A(\meps x{A(x)}) \lif \exists x\, A(x)$ is a standard quantifier
  axiom, and from the critical formula $A(x) \lif A(\meps xA(x))$ we
  get $\exists x\, A(x) \lif A(\meps x{A(x)})$ by the $\exists$-rule,
  since $x$ is not free in $A(\meps xA(x))$.
\end{proof}

\begin{prop}\label{provequiv}
  If $\ql{\Lo}+\Ax$ satisfies the conditions of Lemma~\ref{embed}, then
  $\etx{(\ql{\Lo}+\Ax)} \vdash A \liff A^\et$.
\end{prop}

\begin{proof}
Since $\ql{\Lo}+\Ax$ includes $\ql{\IL}$, the substitution rule $B \liff C
\vdash D(B) \liff D(C)$ is admissible.  The result follows by
induction on complexity of~$A$ and the previous proposition.
\end{proof}

\section{Critical Formulas and Quantifier Shifts}
\label{sec:quant-shift}

We will show later (Theorem~\ref{thm:conservative}) that any
\et-calculus for an intermediate logic~$\Lo$ is conservative
over~$\Lo$. It is well-known that the \et-calculus over intuitionistic
logic is not conservative over intuitionistic \emph{predicate} logic.
We'll show now specifically that for any intermediate logic
$\pl{\Lo}$, the \et-translations of all classically valid quantifier
shift principles are provable from critical formulas.

\begin{table}
  \hrule\vspace{2pt}\hrule
\begin{align*}
  &&& C(x) && t_1 && t_2\\[.5ex] \hline\\[-1.5ex]
%  \forall(\lor)\quad 
%  (\forall x\, A(x) \lor B)  & \lif \forall x(A(x) \lor B) &
%  (A[\mtau xA(x)] \lor B)  & \lif (A[\mtau x(A(x) \lor B)]) \lor B) \\
%  (\forall\lor)\quad &  
  \forall x(A(x) \lor B)  & \lif (\forall x\, A(x) \lor B) &
  & A(x) \lor B && \mtau x C(x) && \mtau x A(x) \tag{\CD}\\
%  (A[\mtau x(A(x) \lor B)] \lor B) &\lif (A[\mtau xA(x)] \lor B) \\
%  \exists(\lor)\quad 
  (\exists x\, A(x) \lor B)  & \lif \exists x(A(x) \lor B) &
  & A(x) \lor B && \meps x A(x) && \meps x C(x)\\
%  (A[\meps xA(x)] \lor B)  & \lif (A[\meps x (A(x) \lor B)] \lor B) \\
%  (\exists \lor)\quad 
%  \exists x(A(x) \lor B)  & \lif (\exists x\, A(x) \lor B) &
%  (A[\meps x (A(x) \lor B)] \lor B) & \lif (A[\meps xA(x)] \lor B) \\
%  \forall(\land)\quad 
%  (\forall x\, A(x) \land B) & \lif \forall x(A(x) \land B) &
% (A[\mtau xA(x)] \land B)  &\lif (A[\mtau x (A(x) \land B)]) \land B) \\
% (\forall\land)\quad 
  \forall x(A(x) \land B)  &\lif (\forall x\, A(x) \land B) &
  & A(x) \land B && \mtau x C(x) && \mtau x A(x)\\
%  (A[\mtau x(A(x) \land B)] \land B) & \lif (A[\mtau xA(x)] \land B) \\
%  \exists(\land)\quad 
  (\exists x\, A(x) \land B)  &\lif \exists x(A(x) \land B) &
  & A(x) \land B && \meps x A(x) && \meps x C(x)\\
%  (A[\meps xA(x)] \land B) & \lif (A[\meps x (A(x) \land B)] \land B) \\
%  (\exists\land)\quad 
%  \exists x(A(x) \land B) & \lif (\exists x\, A(x) \land B) &
%(A[\meps x (A(x) \land B)] \land B) & \lif (A[\meps xA(x)] \land B) \\
%  (\lif)\exists\quad 
  (B \lif \exists x\, A(x)) & \lif \exists x(B \lif A(x)) &
  & B \lif A(x) && \meps x A(x) &&\meps x C(x) \tag{\Eout}\\
%  (B \lif A[\meps xA(x)]) & \lif (B \lif A[\meps x (B \lif A(x))]) \\
%  (\lif\exists)\quad 
%  \exists x(B \lif A(x)) & \lif (B \lif \exists x\,A(x)) &
%  (B \lif A[\meps x (B \lif A(x))])  &\lif (B \lif A[\meps xA(x)]) \\
%  (\lif)\forall\quad 
%  (B \lif \forall x\, A(x)) & \lif \forall x(B \lif A(x)) &
%  (B \lif A[\mtau xA(x)]) & \lif (B \lif A[\tau_x (B \lif A(x))]) \\
%  (\lif\forall)\quad 
  \forall x(B \lif A(x)) & \lif (B \lif \forall x\,A(x)) &
  & B \lif A(x) && \mtau x C(x) && \mtau x A(x)\\
%  (B \lif A[\mtau x (B \lif A(x))]) &\lif (B \lif A[\mtau xA(x)]) \\
%  \exists(\lif)\quad 
  (\forall x\, A(x) \lif B) & \lif \exists x(A(x) \lif B) &
  & A(x) \lif B && \mtau x A(x) && \meps x C(x) \tag{\Aout}\\
%   (A[\mtau xA(x)] \lif B) & \lif (A[\meps x (A(x) \lif B)] \lif B) \\
%  (\forall\lif)\quad 
%  \exists x(A(x) \lif B)  & \lif (\forall x\, A \lif B(x)) &
%  (A[\meps x (A(x) \lif B)]) \lif B)  & \lif (A[\mtau xA(x)] \lif B) \\
%  \forall(\lif)\quad 
%  (\exists x\, A \lif B) & \lif \forall x(A \lif B) &
%  (A[\meps xA(x)] \lif B) & \lif (A[\mtau x (A(x) \lif B)] \lif B) \\
%  (\exists\lif)\quad 
  \forall x(A(x) \lif B) & \lif (\exists x\, A(x) \lif B) &
  & A(x) \lif B && \mtau x C(x) && \meps x A(x)\\
%  (A[\mtau x (A(x) \lif B)] \lif B) & \lif (A[\meps xA(x)] \lif B)
\forall x \lnot\lnot A(x) & \lif \lnot\lnot \forall x\, A(x) &
  & \lnot\lnot A(x) && \mtau x \lnot\lnot A(x) && \mtau x A(x) \tag{\Kur}
\end{align*}
\hrule\vspace{2pt}\hrule

\caption{Quantifier shift formulas whose \et-translations are
  critical formulas. In each case, $x$ is not free in~$B$, and the
  \et-translation of the quantifier shift formula on the left is
  $C(t_1) \lif C(t_2)$.}
\label{qstable}
\end{table}

Quantifier shift formulas divide into two kinds. On the one hand, we
have conditionals the \et-translations of which are critical
formulas, and are therefore provable in $\etl{\pl{\Lo}}$ (see
Table~\ref{qstable}). On the other hand, we have formulas provable
from critical formulas together with some propositional principles,
all of which are intuitionistically valid and hence provable in all
intermediate logics. For instance, to obtain the \et-translation of
\[(\forall x\, A(x) \lor B)   \lif \forall x(A(x) \lor B), \] 
take $A_1 = A(\mtau xA(x))$ and $A_2 =
A(\mtau x{(A(x) \lor B)})$. Then $A_1 \lif A_2$ is a critical formula,
viz.,
\[
A(\mtau xA(x)) \lif A(\mtau x(A(x) \lor B)).
\]
Apply modus ponens to it and the principle 
\[
(A_1 \lif A_2) \lif ((A_1 \lor B) \lif (A_2 \lor B)).
\]
This same pattern works in all cases, the required critical formulas
$A_1 \lif A_2$ and propositional principles are given in
Table~\ref{qsproof}.

\begin{table}
  \hrule\vspace{2pt}\hrule
\[
\begin{array}{r@{}ll}
  (\forall x\,A(x) \lor B) & {} \lif \forall x(A(x) \lor B) &  
 A_1 = A(\mtau xA(x)), A_2 = A(\mtau x(A \lor B)) \\
  \exists x(A(x) \lor B) & {} \lif (\exists x A(x) \lor B) &
 A_1 = A(\meps x (A(x) \lor B)), A_2 = A(\meps xA(x))\\
(A_1 \lor B) & {} \lif (A_2 \lor B) & (A_1 \lif A_2) \lif ((A_1 \lor B) \lif (A_2 \lor B)) \\ \hline\\[-1ex]
  \forall x(A(x) \land B) & {} \lif (\forall x\,A(x) \land B) &  
  A_1 = A(\mtau xA(x)), A_2 = A(\mtau x (A(x) \land B)) \\
  \exists x(A(x) \land B)  & {} \lif (\exists x\, A(x) \land B) &
A_1 = A(\meps x (A(x) \land B)), A_2 = A(\meps xA(x))\\
(A_1 \land B) & {} \lif (A_2 \land B)
& (A_1 \lif A_2) \lif ((A_1 \land B) \lif (A_2 \land B)) \\ \hline\\[-1ex]
\exists x(B \lif A(x)) & {} \lif (B \lif \exists x\,A(x))
& A_1 = A(\meps x (B \lif A(x))), A_2 = A(\meps xA(x))\\
(B \lif \forall x\, A(x)) & {} \lif \forall x(B \lif A(x))
& A_1 = A(\mtau xA(x)), A_2 = A(\mtau x (B \lif A(x))) \\
(B \lif A_1) & {} \lif (B \lif A_2)  & (A_1 \lif A_2) \lif ((B \lif A_1) \lif (B \lif A_2)) \\ \hline\\[-1ex]
\exists x(A(x) \lif B) & {} \lif (\forall x\, A(x) \lif B)
& A_1 = A(\mtau xA(x)), A_2 = A(\meps x (A(x) \lif B)) \\
(\exists x\, A(x) \lif B) & {} \lif \forall x(A(x) \lif B)
& A_1 = A(\mtau x (A(x) \lif B)), A_2 = A(\meps xA(x))\\
(A_1 \lif B) & {} \lif (A_2 \lif B) & (A_1 \lif A_2) \lif ((A_2 \lif B) \lif (A_1 \lif B))
\end{array}\]
\hrule\vspace{2pt}\hrule\vspace*{2ex}

\caption{Proofs of \et-translations of quantifier shift formulas.
In each case, $x$ is not free in~$B$, $A_1 \lif A_2$ is a critical
formula, the \et-translation of the formula is given on the left. The
propositional principle on the right is provable in intuitionistic
logic, and the \et-translation of the quantifier shift formula follows
by one application of modus ponens.}
\label{qsproof}
\end{table}

The most interesting quantifier shift formulas here are $\CD$,
$\Aout$, and $\Eout$, since they are not intuitionistically valid.  By
contrast, we have:

\begin{prop}\label{prop:qsprovet}
  If $\Lo$ is an intermediate propositional logic, then:
  \begin{enumerate}
    \item $K^\et$, $\CD^\et$, $(\Eout)^\et$, and
  $(\Aout)^\et$ are provable in $\etl{\pl{\Lo}}$.

  \item $K$, $\CD$, $\Eout$, and $\Aout$ are provable in
  $\etx{(\ql{\Lo}+\Ax)}$.
  \end{enumerate}
\end{prop}

\begin{proof}
(1) They are critical formulas; see Table~\ref{qstable}.

(2) Follows from Proposition~\ref{provequiv}.
\end{proof}

The second \eps-theorem states that if $A^\eps$ is provable in the
pure \eps-calculus, then $A$ is provable in classical predicate logic.
The second \eps-theorem fails for any intermediate predicate logic
$\ql\Lo +\Ax$, in which $A^\et$ is provable in $\etl\Lo$ but $A$ is
not provable in $\ql\Lo +\Ax$, e.g., when $\ql\Lo +\Ax$ does not prove
one of $K$, $\CD$, $\Eout$, or~$\Aout$.

Note that the only intuitionistically invalid De Morgan rule for
quantifiers, 
\begin{equation*}
  \lnot \forall x\,A(x) \lif \exists x\,\lnot A(x), \tag{$Q$}
\end{equation*} 
is a special case of $\Aout$, taking $\bot$ for~$B$; $Q^\et$ is a
critical formula. The \et-translation of double negation shift
$\Kur^\et$ is 
\[\lnot\lnot A(\mtau x \lnot\lnot A(x)) \lif \lnot\lnot A(\mtau x A(x))\] 
and is also a critical formula.

In classical first-order logic, both the addition of \eps-operators
and critical formulas and the replacement of quantifiers by
\eps-operators is conservative. The previous results show that for
extensions of first-order intuitionistic logic, this is not the case:
intuitionistically invalid quantified formulas (or their
\et-translations) become provable. However, these quantifier shifts
\emph{are} provable in some intermediate logics, e.g., in some G\"odel
logics.

We might think of \et-terms semantically as terms for objects which
serve the role of generics taking on the role of quantifiers, and
indeed in classical logic this connection is very close. Because of
the validity of 
\begin{align*}
  & \exists x(\exists y\,A(y) \lif A(x)) \tag{\Wel}\\
  & \exists x(A(x) \lif \forall y\,A(y)) \tag{\DWel}
\end{align*}
in classical logic, there always is an object~$x$ which behaves as an
\eps-term ($A(x)$ holds iff $\exists x\,A(x)$ holds), and an
object~$x$ which behaves as a \tau-term (i.e., $A(x)$ holds iff
$\forall y\,A(y)$ holds). One might expect then that $\Wel$ and
$\DWel$, when added to $\ql{\IL}$, have the same effect as adding
critical formulas, i.e., that all quantifier shifts become provable.
Note that $\Wel$ and $\DWel$ are intuitionistically equivalent to
\begin{align*}
  & \exists x\forall y(A(y) \lif A(x)) \tag{$\Wel'$}\\
  & \exists x\forall y(A(x) \lif A(y)). \tag{$\DWel'$}
\end{align*}
As is easily checked, $\proves{\ql{\IL}} \Wel \liff \Eout$ and
$\proves{\ql{\IL}} \DWel \liff \Aout$, and even $\proves{\ql{\IL}}
\DWel \lif \CD$. However, $\ql{\IL} + \Wel \nvdash \CD$.\footnote{See
p.~694 of \citet{Skvortsov2006}.}

\section{$\etl{\pl{\Lo}}$ is Conservative over $\pl{\Lo}$}
\label{sec:conservative}

The classical \eps-calculus is conservative over propositional logic.
Work by \citet{Bell1993a} and \citet{DeVidi1995} shows that, however,
the addition of critical formulas to intuitionistic logic results in
intuitionistically invalid \emph{propositional} formulas becoming
provable in certain simple theories. These results require the
presence of identity axioms.  One may wonder if these results can be
strengthened to the pure logic and the \et-calculus alone. The
following proposition shows that this is not the case. The addition of
critical formulas to intermediate logics alone does not have any
effects on the propositional level.

\begin{defn}
  The \emph{shadow} $A^s$ of a formula is defined as follows:
  \begin{align*}
    P(t_1, \ldots, t_n)^s & = X_P\\
    (t_1 = t_2)^s & = \top\\
    (A \land B)^s & = A^s \land B^s & (A \lor B)^s & = A^s \lor B^s \\
    (A \lif B)^s & = A^s \lif B^s & (\lnot A)^s & = \lnot A^s \\
    (\exists x\, A(x))^s & = A(x)^s & (\forall x\, A(x))^s & = A(x)^s
  \end{align*} 
  where $X_P$ is a propositional variable and $\top$ is any theorem
  of~$\pl{\Lo}$.

  The shadow of a proof $\pi = A_1$, \dots, $A_n$ is $A_1^s$, \dots,
  $A_n^s$.

  A first-order intermediate logic $\ql{\Lo} + \Ax$ (over a
  propositional base logic~$\pl{\Lo}$) is \emph{preserved under
  shadow} if $\proves{\pl{\Lo}} B^s$ for all quantifier axioms $B \in
  \Ax$.
\end{defn}

The shadow of a formula is a propositional formula obtained by
disregarding all first-order structure. If an intermediate predicate
logic $\ql{\Lo} + \Ax$ is preserved under shadow, the shadows of its
theorems are already valid in~$\pl{\Lo}$.

\begin{prop}\label{shadow} Suppose $\ql{\Lo}+\Ax$ is preserved under
  shadow. If $A_1, \ldots, A_n \vdash_{\etx{(\ql{\Lo} + \Ax)}} B$,
  then $A_1^s, \ldots, A_n^s \vdash_{\pl{\Lo}} B^s$. This also holds
  if identity axioms are present.
\end{prop}

\begin{proof}
Consider a derivation~$\pi$ in $\etx{(\ql\Lo+\Ax)}$ of $B$ from $A_1$,
\dots,~$A_n$, and a formula~$C$ in~$\pi$ not a conclusion of an
inference, and not among $A_1$, \dots,~$A_n$. If $C \in \Lo$, then
also $C^s \in \Lo$.  If $C$ is a critical formula, then $C^s$ is of
the form $A \lif A$. If $C$ is a standard quantifier axiom, we have
$(\forall x\, A \lif A(t))^s \equiv (A(t) \lif \exists x\, A(x))^s
\equiv A(x)^s \lif A(x)^s$, which again is in~$\Lo$. (Clearly, $A(t)^s
\equiv A(x)^s$.) 
  
If $C$ is the conclusion of modus ponens from premises $A$ and $A \lif
C$, then $C^s$ follows from $A^s$ and $(A \lif C)^s$ by modus ponens.
If $C$ is the conclusion of a quantifier rule, the shadows of premise
and conclusion are identical, e.g.,
  \[
  (B \lif A(x))^s \equiv B^s \lif A(x)^s \equiv (B \lif \forall x\, A(x))^s
  \]
Thus we have shown that $\pi^s$ is a derivation of $B^s$ from $A_1^s$,
\dots, $A_n^s$ in~$\Lo$.

This still holds if identity is present, as the shadows of identity
axioms are: $(t = t)^s \equiv \top$ and
  \[
  (t_1 = t_2 \lif (A(t_1) \lif A(t_2)))^s \equiv \top \lif (A(t_1)^s \lif A(t_1)^s)
  \]
since $A(t_1)^s \equiv A(t_2)^s$. Both are provable in~\IL{} and thus
in~$\Lo$.
\end{proof}

All intermediate predicate logics mentioned above are preserved under
shadow. They are axiomatized by various quantifier shift principles.
As we have seen in the preceding section, the \et-translations of all
such quantifier shift principles become provable in the corresponding
\et-calculus. However, the shadow of such a quantifier shift principle
is a formula of the form $B \lif B$. As a consequence, we have the
following conservativity result for all intermediate \et-calculi:

\begin{thm}\label{thm:conservative} If $\ql{\Lo} + \Ax$ is preserved
  under shadow, then $\etx{(\ql{\Lo}+\Ax)}$ is conservative over
  $\pl{\Lo}$ for propositional formulas. In particular, no new
  propositional formulas become provable in~$\etl{\pl{\Lo}}$ by the
  addition of critical formulas to any intermediate logic~$\pl{\Lo}$,
  including intuitionistic logic itself.
\end{thm}

\citet{Bell1993a} claimed that in the extended intuitionistic
\eps-calculus for $\ql{\IL}$ with identity, we have $D
\vdash_{\etx{\ql{\IL}}} M$, where $M$ is
\begin{equation*}
 \lnot(B \land C) \lif (\lnot B \lor \lnot C) \tag{$M$}
\end{equation*}
and $D$ is $\forall x(x = a \lor \lnot x =a)$. $M$ is an
intuitionistically invalid direction of De~Morgan's laws. Since the
shadow~$D^s$ of $D$ is the intuitionistically valid formula $\top
\lor \lnot \top$, this seems to contradict Lemma~\ref{shadow}.
The proof starts by asserting that 
\[ 
  \forall x\,[(x = a \land B) \lor (x \neq a \land C)]  \lif (B
  \land C)
\]
is provable in $\ql{\IL}$ with identity. This is false, however, as
the formula is not true in any one-element model when $B$~is true and
$C$~is false. Theorem~7 of \citealt{DeVidi1995} fails for the same
reason. The results are correct with the additional assumption $a \neq
b$.\footnote{Bell provides another proof of $M$ in intuitionistic
\eps-calculus which explicitly requires, in addition to~$D$, the assumption $a \neq b$. \citet{DeVidi1995} shows that in the
intuitionistic \et-calculus, $D \land a \neq b$ derives~\LIN. (Note
that also $\proves{\LC} M$.) However, since the shadow of $a \neq b$
is $\lnot \top$, these proofs do not conflict with our
Lemma~\ref{shadow}.  Bell's other examples of intuitionistically
invalid propositional formulas provable in \et-calculi all require
assumptions of the form $a \neq b$ and also the axiom of
\eps-extensionality.  The examples of derivations of $M$ and~\LIN{} in
intuitionistic \et-calculus given by \citet{Mulvihill2015} avoid
identity but require the assumptions $\forall x((P(x) \lif P(a)) \lor
\lnot(P(x) \lif P(a)))$ and $\lnot(P(a) \lif P(b))$.}

\section{The Extended First \et-Theorem Fails unless $\proves{\Lo} B_m$}
\label{sec:negative}

In classical first-order logic, the main result about the
\eps-calculus is the extended first \eps-theorem. It states that if
$A(e_1, \dots, e_n)$, where the $e_i$ are \eps-terms, is provable in
the pure \eps-calculus, then there are \eps-free terms $t_{i}^{j}$ such that
\[
A(t_{i}^{1}, \dots, t_{n}^{1}) \lor \dots \lor A(t_{i}^{k}, \dots, t_{n}^{k})
\]
is provable in classical propositional logic alone. Such \eps-terms
$e_i$ appear as the result of translating $\exists x_1\dots\exists
x_n\,A(x_1, \dots, x_n)$ into the \eps-calculus. 

In the context of intermediate logics, we may formulate the statement
as follows:

\begin{defn}\label{def:firstet} An intermediate logic $\pl{\Lo}$
\emph{has the extended first \et-theorem}, if, whenever
$\proves{\etl{\pl{\Lo}}} A(e_1, \dots, e_n)$ for some \eps- or
\tau-terms $e_1$, \dots, $e_n$, then there are \et-free terms
$t_{i}^{j}$ such that 
\[
\proves{\pl{\Lo}} A(t_{i}^{1}, \dots, t_{n}^{1}) \lor \dots \lor A(t_{i}^{k}, \dots, t_{n}^{k}).
\]
\end{defn}

We obtain a first \emph{negative} result: if $\etl{\pl{\Lo}}$ has the
extended first epsilon theorem, then an instance of~$\Bm m$, i.e., 
\[ 
(A_1 \lif A_2) \lor \ldots \lor (A_{m} \lif A_{m+1})
\] 
for some $m \ge 2$ is provable already in the propositional
fragment~$\pl{\Lo}$.\footnote{For $m=2$, this schema is equivalent to
$A \lor \lnot A$: take $\top$ for $A_1$, $A$ for $A_2$, $\bot$ for
$A_3$.}  This rules out an extended first \et-theorem for, e.g.,
\et-calculi for intuitionistic logic and infinite-valued
G\"odel-Dummett logic.

\begin{thm}\label{thm:forking}
  Suppose $\etl{\pl{\Lo}}$ has the extended first \et-theorem. Then
  $\proves{\pl{\Lo}} \Bm m$ for
  some~$m \ge 2$.
\end{thm}

\begin{proof}
Consider Let $A(z) \equiv (P(f(z)) \lif P(z))$ and $\exists
z\, A(z)$, i.e., $\exists z(P(f(z)) \lif P(z))$. Let $e
\equiv \meps z(P(f(z)) \lif P(z))$. The \et-translation of
$\exists z\, A(z)$ is is $P(f(e)) \lif P(e)$, i.e.,
\[
V \equiv P(f(\meps z(P(f(z)) \lif P(z)))) \lif P(\meps z(P(f(z))
\lif P(z)))
\]
Let $U \equiv A(\meps xP(x)) \equiv P(f(\meps xP(x))) \lif P(\meps
xP(x))$.  Note that $U$ is of the form $P(t) \lif P(\meps xP(x))$, so
it is a critical formula.  Also note that $U \lif V$ is of the form
$A(t) \lif A(e)$, and so $U \lif V$ is also a critical formula.

Since $\proves{\pl{\Lo}} (U \lif V) \lif (U \lif V)$, and $U \lif V$
and $U$ are critical formulas, $\proves{\etl{\pl{\Lo}}} V$. By
assumption, $\etl{\pl{\Lo}}$ has the extended first \et-theorem, so
$\pl{\Lo}$ proves a disjunction of the form
\[
(P(f(t_1)) \lif P(t_1)) \lor \dots \lor (P(f(t_k)) \lif P(t_k))
\]
for some terms $t_1$, \dots, $t_k$. (This is a Herbrand disjunction of
$\exists z\, A(z)$.) Each term $t_i$ is of the form~$f^i(s)$ for some
$i \ge 0$ and a term~$s$ which does not start with~$f$.  By
rearranging the disjuncts to group disjuncts with the same innermost
term~$s$ together (using commutativity of~$\lor$) and by adding
additional disjuncts as needed (using weakening), from this we obtain
a formula
\begin{align*}
  (P(f^{j_1+1}(s_1))
  \lif P(f^{j_1}(s_1))) \lor {} & \dots \lor (P(f(s_1)) \lif P(s_1)) \lor \\ & \vdots \\ 
  (P(f^{j_l+1}(s_l)) \lif P(f^{j_l}(s_l))) \lor {}  & \dots \lor (P(f(s_l)) \lif P(s_l))
\end{align*}
Let $j$ be the largest among $j_1$, \dots,~$j_l$. By uniformly
replacing $P(f^i(s_j))$ by $A_{j+2-i}$ in the proof of the last
formula and contracting identical disjuncts, we obtain a proof
in~$\Lo$ of $(A_1 \lif A_2) \lor \dots \lor (A_{j+1} \lif A_{j+2})$.
This is $\Bm m$ for $m=j+1$, and since $j\ge 1$, $m\ge 2$.
\end{proof}

A formula of the form~$\Bm m$ is provable in~$\pl{\Lo}$ iff $\pl{\Lo}$
is a finite-valued G\"odel logic~$\LC_n$
(Proposition~\ref{prop:fork-finite}). By contrast, no~$\Bm m$ is
provable in intuitionistic logic~\IL, Jankov logic~$\KC$, or in
infinite-valued G\"odel logic~$\LC$ (Proposition~\ref{prop:Bm-Gmetc}).

\begin{prop}\label{prop:Bm-Gmetc}
\begin{enumerate}
  \item $\proves{\LC_m} \Bm m$
  \item $\pl{\Lo} \not\vdash \Bm n$ for $\Lo$ any of $\LC_m$ with
  $m>n$, $\LC$, $\KC$, $\IL$.  
\end{enumerate}
\end{prop}

\begin{proof}
  (1) Follows by definition, since $\LC_m = \LC + \Bm m$.

  (2) Let $v(A_i) = 1/i$ if $i < n$ and $v(A_{n+1}) = 0$. This is a
  valuation in a truth value set with $m$ elements if $m > n$ (i.e., a
  valuation in the G\"odel semantics for~$\LC_m$). It is also a
  valuation in the infinite truth value set~$[0,1]$ of~$\LC$. For all
  $i \le n$, $v(A_i) > v(A_{i+1})$ and hence $v(\Bm n) < 1$. So $\Bm
  n$ is not a tautology of $\LC_m$ or~$\LC$. Since $\IL \subsetneq \KC
  \subsetneq \LC$, the result also follows for \KC{} and~\IL.
\end{proof}

\begin{prop}\label{prop:Bn-lin}
  $\IL + \Bm m \vdash \LIN$
\end{prop}

\begin{proof}
  Simultaneously substitute $A$ for $A_i$ if $i$ is odd, and $B$ for
  $A_i$ if $i$ is even in $\Bm m$. The result is one of
  \begin{align*}
    & (A \lif B) \lor (B \lif A) \lor \dots \lor (A \lif B)\\
    & (A \lif B) \lor (B \lif A) \lor \dots \lor (B \lif A)
  \end{align*}
  Both are equivalent in \IL{} to $(A \lif B) \lor (B \lif A)$.
\end{proof}

\begin{prop}\label{prop:fork-finite}
  If $\pl{\Lo} \vdash \Bm n$, then $\pl{\Lo} = \LC_m$ for some~$m$.
\end{prop}

\begin{proof}
  \citet{Hosoi1966} showed that the $n$-valued G\"odel logic is
  axiomatized by $\IL + R_{n-1}$, where $R_{n}$ is
  \[
  A_1 \lor (A_1 \lif A_2) \lor \dots \lor (A_{n-1} \lif A_n) \lor
  \lnot A_n.
  \]
  By simultaneously substituting $\top$ for $A_1$, and $\bot$ for
  $A_{n+1}$, and $A_{i-1}$ for $A_{i}$ ($i = 2$, \dots, $n$) in $\Bm n$,
  we obtain
  \[
  (\top \lif A_1) \lor (A_1 \lif A_2) \lor \dots \lor (A_{n-2} \lif
  A_{n-1}) \lor (A_{n-1} \lif \bot),
  \]
  which is equivalent to~$R_{n-1}$ in \IL. Hence, since $\pl{\Lo} \vdash
  \Bm n$, $\LC_n \subseteq \pl{\Lo}$.

  Furthermore, \citet[Lemma 4.1]{Hosoi1967a} showed that if $\Lo
  \vdash \LIN$ then $\Lo = \LC_m$ for some $m$ or $L = \LC$. Since
  $\pl{\Lo} \vdash \Bm n$, $\pl{\Lo} \vdash \LIN$ by
  Proposition~\ref{prop:Bn-lin}. The result follows as $\LC \nvdash
  \Bm n$ and so $\pl{\Lo} \neq \LC$.\footnote{In Hosoi's nomenclature,
  $\LC_n$ is $\mathbf{S}_{n-1}$ and $\LC$ is $\mathbf{S}_\omega$.}
\end{proof}

\begin{cor}\label{cor:negative} No intermediate logic except $\LC_m$
  has the extended first \et-theorem. In particular, intuitionistic
  logic~\IL, Jankov logic~\KC, and infinite-valued G\"odel logic~\LC{}
  do not have the extended first \et-theorem.
\end{cor}

We have restricted $\pl{\Lo}$ here to be an intermediate propositional
logic. However, it bears remarking that Theorem~\ref{thm:forking} does
not require that $\pl{\Lo}$ contains~$\IL$. An inspection of the proof
shows that all that is required is that $\proves{\pl{\Lo}} A \lif A$,
and in~$\pl{\Lo}$, $\lor$~is provably commutative, associative, and
idempotent, and has weakening ($\proves{\pl{\Lo}} A \lif (A \lor B)$).
Thus, Corollary~\ref{cor:negative} applies to any \et-calculus based
on a logic which has these properties (such as, say, \L ukasiewicz
logic.)

The extended first \eps-theorem in classical logic shows that if an
existential formula~$\exists x\,A(x)$ is provable, so is a disjunction
of instances $\bigvee_i A(t_i)$. Clearly this is equivalent to: if
$\forall x\,A(x) \lif B$ is provable so is $\bigwedge A(t_i) \lif B$.
Without the interdefinability of $\forall$ and~$\exists$, the question
arises whether the alternative form of the \eps-theorem might hold in
an intermediate \et-calculus even if the standard form does not. We'll
show that the versions are, in fact, equivalent even in intermediate
logics.

\begin{prop}\label{prop:eps-thm-versions}
  The following are equivalent:
  \begin{enumerate}
    \item If $\proves{\etl{\pl{\Lo}}} A(e)$ then $\proves{\pl{\Lo}} \bigvee_i A(t)$.
    \item If $B(e') \vdash_{\etl{\pl{\Lo}}} C$ then $\bigwedge_j B(s_j)
    \vdash_{\pl{\Lo}} C$ for $C$ \et-free.
    \item If $B(e') \vdash_{\etl{\pl{\Lo}}} C(e)$ then $\bigwedge_j B(s_j)
    \vdash_{\pl{\Lo}} \bigvee_i C(t_i)$.
  \end{enumerate}
\end{prop}

\begin{proof}
  (1) implies (3): Suppose  
  \begin{align*}
    B(e') & \vdash_{\etl{\pl{\Lo}}} C(e). 
    \intertext{By the
  deduction theorem, }
  & \vdash_{\etl{\pl{\Lo}}} B(e') \lif C(e). 
  \intertext{By (1) we
  have terms $s_i$, $t_i$ so that}
  & \vdash \bigvee_i (B(s_i) \lif C(t_i))
  \intertext{By intuitionistic logic,}
  & \vdash  \bigwedge_i B(s_i) \lif \bigvee_i C(t_i) \text{ and so}\\
  \bigwedge_i B(s_i) & \vdash  \bigvee_i C(t_i)
\end{align*}
by the deduction theorem.

(3) clearly implies (1) and (2).

(2) implies (1): Let $X$ be a propositional variable. $A \vdash_\IL (A
\lif X) \lif X$. So if $\proves{\etl{\pl{\Lo}}} A(e)$ then by the deduction
theorem,
\begin{align*}
  A(e) \lif X & \vdash_{\etl{\pl{\Lo}}} X \text{ and by (2),}\\
  \bigwedge_i (A(s_i) \lif X) & \vdash_{\pl{\Lo}} X.
  \intertext{Now substitute $\bigvee_i A(s_i)$ for $X$:}
  \bigwedge_i (A(s_i) \lif \bigvee_i A(s_i)) & \vdash_{\pl{\Lo}} \bigvee_i A(s_i).
\end{align*}
The formula on the left is provable intuitionistically.
\end{proof}

\section{Elimination Sets and Excluded Middle}
\label{sec:elim-sets}

The basic idea of Hilbert's proof of the extended first \eps-theorem
is this: Suppose we have a proof of $E \equiv D(e)$ from critical
formulas $\Gamma, \Lambda(e)$, where $e$ is a critical $\eps$-term and
$\Lambda(e)$ is a set of critical formulas belonging to~$e$. Now we
find terms $t_1$, \dots, $t_k$ such that replacing $e$ by $t_i$ allows
us to remove the critical formulas $\Lambda(e)$, while at the same
time replacing the end-formula $D(e)$ by $\bigvee_{i=1}^k D(t_i)$ and
the remaining critical formulas by $\Gamma[t_1/e]$, \dots,
$\Gamma[t_k/e]$. We repeat this procedure in such a way that
eventually all critical formulas are removed and we are left with a
disjunction of instances of~$E$, as required by the first
\eps-theorem.  The difficulty of making this work lies in three
challenges. The first is to find a suitable way of selecting
\eps-terms~$e$ and corresponding critical formulas~$\Lambda(e)$ so
that the~$\Lambda(e)$ can be removed. The second is to ensure that in
passing from $\Gamma$ to~$\Gamma[t_i/e]$ we again obtain critical
formulas.\footnote{Replacing an \eps-term in a critical formula by another
term does in general not result in a critical formula. E.g., let $A(y)
\equiv B(\meps x C(x, y), y)$ and $e\equiv \meps x C(x, t)$ then
$A(t)[s/e]$ is $B(s, t)$ but $A(\eps y A(y))[s/e]$ is just $A(\eps y
A(y))$.} The third challenge is to guarantee that the process
eventually terminates with no critical formulas remaining.

In this section we address the first challenge by considering the
condition that suffices to overcome it: the existence of complete
$e$-elimination sets (defined below) for every \et-term~$e$. We then
show why this condition is satisfied in classical logic, so we can
clarify the role of excluded middle in the proof for the classical
case, as well as how the proof for classical logic and those for
intermediate logics given later correspond to one another. We will
discuss the condition for intermediate logics in
Section~\ref{sec:positive} and the remaining challenges in
Section~\ref{sec:hb-elim}.

\begin{defn}
  Suppose $\Gamma \vdash_{\etl{\Lo}}^\pi D$ with critical
  formulas~$\Gamma$, and $e$ is an \eps-term $\meps x A(x)$ (\tau-term
  $\mtau x A(x)$). If $C\equiv A(t) \lif A(\meps x A(x)) \in \Gamma$
  ($C \equiv A(\mtau x A(x)) \lif A(t) \in \Gamma$) we say $e$ is
  \emph{the critical \et-term of $C$}, that $e$ \emph{belongs to}~$C$,
  and that $e$ is \emph{a critical
  \et-term} of~$\pi$.
\end{defn}
  
\begin{defn}
  Suppose $\Gamma, \Lambda(e), \Lambda'(e)
  \vdash_{\etl{\pl{\Lo}}}^\pi D(e)$ where $\Lambda(e) \cup \Lambda(e)'$ are
  all critical formulas belonging to~$e$. A set of terms $s_1$, \dots, $s_k$ is an
  \emph{$e$-elimination set} for $\pi$ and $\Lambda(e)$ if
  \begin{align*}
    \Gamma[s_1/e], \dots, \Gamma[s_k/e], \Lambda'(e) & \vdash_{\pl{\Lo}} 
    D(s_1) \lor \dots \lor D(s_k).
  \end{align*}
  If $\Lambda(e)$ is the set of all critical formulas belonging to~$e$
  (i.e., $\Lambda'(e) = \emptyset$)
  then an $e$-elimination set for $\Lambda(e)$ is called
  a \emph{complete $e$-elimination set}.
\end{defn}

Here, $\Gamma[s_i/e]$ means the result of replacing, in each formula
in $\Gamma$, every occurrence of $e$ by~$s_i$. If $T = \{s_1, \dots,
s_k\}$ we write $\Gamma[T]$ for $\Gamma[s_1]$, \dots, $\Gamma[s_k]$.
Note that we do not require in the definition of $e$-elimination sets
that the formulas in $\Gamma[s_i/e]$ are actually critical formulas.

\begin{lem}
  If $C \equiv A(t) \lif A(e)$ or $C \equiv A(e) \lif A(t)$ is a
  critical formula with critical \et-term~$e$, then $C[s/e]$ is
  $A(t[s/e]) \lif A(s)$ or $A(s) \lif A(t[s/e])$, respectively.
\end{lem}

\begin{proof}
  Since $e$ is the critical \eps-term of~$C$, $e \equiv \meps x A(x)$
  or $e \equiv \mtau x A(x)$. Hence, $e$ cannot occur in~$A(x)$, since
  otherwise it would be a proper subbterm of itself.
\end{proof}

\begin{lem}
  If $\Gamma \vdash_{\Lo} D$ then $\Gamma[t/e] \vdash_{\Lo} D[t/e]$
\end{lem}

\begin{proof}
  Any proof of $D$ from $\Gamma$ using modus ponens and axioms
  of~$\pl{\Lo}$ remains correct if terms in it are uniformly replaced by
  other terms.
\end{proof}

\begin{lem}\label{lem:facts}
  In any intermediate logic~$\Lo$:
  \begin{enumerate}
  \item \label{disj} If $\Gamma, A \vdash C$ and $\Gamma', B
  \vdash_\Lo D$ then $\Gamma, \Gamma', A \lor B \vdash C \lor D$.
  \item \label{impl} If $\Gamma, A \vdash_\Lo C$ and $B \vdash A$, then $\Gamma, B
  \vdash C$.
  \end{enumerate}
\end{lem}

We are now in a position to apply the preceding lemmas and the concept
of elimination sets to the case of classical logic. This elucidates
how the first challenge is solved in the proof of the extended first
\et-theorem for classical logic where \tau-terms and corresponding
critical formulas may also be present. (For Hilbert's original proof
for the \eps-calculus without \tau-terms, see
\cite{HilbertBernays1939} or \cite{MoserZach2006}.)

\begin{prop}\label{prop:elimCLsingle}
  In $\etl{\CL}$, every critical formula $C(e)$ has an $e$-elimination
  set.
\end{prop}

\begin{proof}
  Suppose first that $e$ is an $\eps$-term; then $C(e)$ is $A(s) \lif
  A(e)$. Let $\Lambda'(e)$ be the critical formulas belonging to~$e$
  other than~$C(e)$, and $\Gamma$ the remaining critical formulas for
  which $e$ is not critical. So we have:
  \begin{align*}
    \Gamma, \Lambda'(e), A(s) \lif A(e) & \vdash_{\CL} D(e)
    \intertext{On the one hand, by replacing $e$ everywhere by~$s$ we get}
    \Gamma[s/e], \Lambda'(s), A(s[s/e]) \lif A(s) & \vdash_{\CL} D(s)
    \intertext{and by Lemma~\ref{lem:facts}(\ref{impl}), since $A(s) 
    \vdash A(t[s/e]) \lif A(s) \in \Lambda'(s)$ and $A(s) \vdash C(s)$,}
    \Gamma[s/e], A(s) & \vdash_{\CL} D(s)
    \intertext{On the other hand, since $\lnot A(s) \vdash A(s) \lif A(e)$,}
    \Gamma, \lnot A(s) & \vdash_{\CL} D(e) \text{ and so,}\\
    \Gamma, \Gamma[s/e], A(s) \lor \lnot A(s)& \vdash_{\CL} D(e) \lor D(s). 
    \intertext{by Lemma~\ref{lem:facts}(\ref{disj}). Since $\proves{\CL} A(s) \lor \lnot A(s)$ we have}
    \Gamma, \Gamma[s/e] & \vdash_{\CL} D(e) \lor D(s)
  \intertext{Thus, $\{e, s\}$ is an $e$-elimination set for the critical
  formula~$C(e)$.
  \endgraf
  Similarly, if $e$ is a $\tau$-term and $C(e)$ is $A(e) \lif A(s)$ we
  get}
    \Gamma[s/e], \Lambda'(s), C(s) & \vdash_{\CL} D(s)
    \intertext{and by Lemma~\ref{lem:facts}(\ref{impl}), since $\lnot A(s) 
    \vdash C' \in \Lambda'(e)$ and $\lnot A(s) \vdash C(s)$,}
    \Gamma[s/e], \lnot A(s) & \vdash_{\CL} D(s)
    \intertext{On the other hand, $A(s) \vdash C(e)$, so}
    \Gamma, A(s) & \vdash_{\CL} D(e) \text{ and}\\ 
    \Gamma, \Gamma[s/e], \lnot A(s) \lor A(s)& \vdash_{\CL} D(e) \lor D(s).
  \end{align*}
  by Lemma~\ref{lem:facts}(\ref{disj}).
\end{proof}

More generally, the set of all critical formulas belonging to~$e$ has
a (complete) $e$-elimination set in $\CL$:

\begin{prop}\label{prop:elimCL}
  In $\etl{\CL}$, every critical $\et$-term has a complete $e$-elimination
  set.
\end{prop}

\begin{proof}
  Let $C_1 \equiv A(s_1) \lif A(e)$, \dots, $C_k \equiv A(s_k) \lif
  A(e)$ be the critical formulas belonging to~$e$ if $e$ is an
  \eps-term. Since
  \begin{align*}
    \Gamma, C_1(e), \dots, C_k(e) & \vdash_\CL D(e) \text{, also}\\
    \Gamma[s_i/e], C_1(s_i), \dots, C_k(s_i) & \vdash_\CL D(s_i)
  \intertext{(writing $C_j(s_i)$ for $C_j[s_i/e]$). Since $A(s_i) \vdash 
  A(s_j(s_i)) \lif A(s_i) \equiv C_j(s_i)$,}
    \Gamma[s_i/e], A(s_i) & \vdash_\CL  D(s_i)
  \intertext{by Lemma~\ref{lem:facts}(\ref{impl}). By applying Lemma~\ref{lem:facts}(\ref{disj}),}
   \Gamma[s_1/e], \dots, \Gamma[s_k/e], A(s_1) \lor \dots \lor A(s_k) 
   & \vdash_\CL D(s_1) \lor \dots \lor D(s_k).
  \intertext{On the other hand, since $\lnot A(s_i) \vdash A(s_i) \lif A(e)$, we get $\lnot A(s_1) \land \ldots \land A(s_k)
  \vdash C_j(e)$ for each $j = 1$, \dots, $k$, so we also have, from 
  the first line by Lemma~\ref{lem:facts}(\ref{impl}),}
   \Gamma, \lnot A(s_1) \land \ldots \land \lnot A(s_k) & \vdash_\CL D(e).
   \intertext{Since $\bigvee_i A(s_i) \lor (\bigwedge_i \lnot A(s_i))$ 
   is an instance of excluded middle, we have}
   \Gamma, \Gamma[s_1/e], \dots, \Gamma[s_k/e] & \vdash_\CL 
   D(e) \lor D(s_1) \lor \dots \lor D(s_k).
  \end{align*}

  If $e$ is a \tau-term, then the critical formulas are of the form
  $C_j(e) \equiv A(e) \lif A(s_j(e))$ and consequently $C_j(s_i)$ is
  $A(s_i) \lif A(s_j(s_i))$. Each is implied by $\lnot A(s_i)$, so we
  have
  \begin{align*}
    \Gamma[s_1/e], \dots, \Gamma[s_k, e], \lnot A(s_1) \lor \dots \lor \lnot A(s_k) 
    & \vdash_\CL D(s_1) \lor \dots \lor D(s_k)\\
  \intertext{On the other hand, $A(s_1) \land \dots \land A(s_k) \vdash A(e) \lif A(s_j)$, so}
  \Gamma, A(s_1) \land \dots \land A(s_k) & \vdash_\CL  D(e)
  \intertext{and consequently}
  \Gamma, \Gamma[s_1/e], \dots, \Gamma[s_k/e] 
  & \vdash_\CL D(e) \lor D(s_1) \lor \dots \lor D(s_k),
  \end{align*}
  since $\bigwedge_i A(s_i) \lor \bigvee_i \lnot A(s_i)$ is a
  tautology.
  
  In each case, $\{e, s_1, \dots, s_k\}$ is an $e$-elimination set.
\end{proof}

\begin{rem}\label{rem:one-by-one}
  Of course, the fact that in \CL{} we have complete $e$-elimination
  sets can also be obtained by applying
  Proposition~\ref{prop:elimCLsingle} $k$-many times. Applying it to
  $C_1(e)$ results in $T_1 = \{e, s_1(e)\}$, applying it to $C_2(e)$
  in $T_2 = \{e, s_1(e), s_2(e), s_1(s_2(e))\}$, to $C_3$ in $T_3 =
  \{e, s_1(e), s_2(e), s_1(s_2(e)), s_3(e), s_1(s_3(e)), s_2(s_3(e)),
  s_1(s_2(s_3(e)))\}$, etc., i.e., the resulting disjunction has
  $2^{k+1}$ disjuncts, whereas the disjunction resulting from
  Proposition~\ref{prop:elimCL} only has $k+1$ disjuncts. However, see
  Remark~\ref{rem:speedup}.
\end{rem}

We know that intermediate logics other than $\LC_m$ do not have the
extended first \et-theorem and so not every $\et$-term will have complete
$e$-elimination sets. However, if the starting formula $E$ is of a
special form, they sometimes do. In the proof for the classical case
above, this required excluded middle. But it need not. For instance,
if $E$ is negated, then weak excluded middle ($\lnot A \lor \lnot\lnot
A$) is enough.

\begin{prop}\label{prop:elimJ}
  If $\Lo \vdash \J$, then every \et-term in an $\etl{\Lo}$-proof of
  $\bigvee_j \lnot D_j$ has a complete $e$-elimination set.
\end{prop}

\begin{proof}
  Let $C_1 \equiv A(s_1) \lif A(e)$, \dots, $C_k \equiv A(s_k) \lif
  A(e)$ be the critical formulas belonging to $e$ if $e$ is an
  \eps-term. As before, we have
  \begin{align*}
    \Gamma[s_i/e], A(s_i) & \vdash_\Lo  \bigvee_j \lnot D_j(s_i)
  \intertext{In $\KC$, $B \lif (\lnot C_1 \lor \lnot C_2) \vdash 
  \lnot\lnot B \lif (\lnot C_1 \lor \lnot C_2)$, so}
    \Gamma[s_i/e], \lnot\lnot A(s_i) & \vdash_\Lo \bigvee_j \lnot D_j(s_i)
    \intertext{We obtain}
    \Gamma[s_1/e], \dots, \Gamma[s_k/e], &\\
     \lnot \lnot A(s_1) \lor \dots \lor \lnot\lnot A(s_k) 
    & \vdash_\Lo \bigvee_j \lnot D_j(s_1) \lor \dots \lor \bigvee_j \lnot D_j(s_k)
    \intertext{Again as before, we have}
  \Gamma, \lnot A(s_1) \land \ldots \land \lnot A(s_k) & \vdash_\Lo \bigvee_j \lnot D_j(e),
  \intertext{and together}
  \Gamma[s_1/e], \dots, \Gamma[s_k/e], \Gamma, &\\
   (\lnot A(s_1) \land \ldots \land \lnot A(s_k))
   \lor {} &\\ 
   \lnot \lnot A(s_1) \lor \dots \lor \lnot\lnot A(s_k) 
    & \vdash_\Lo  \bigvee_j \lnot D_j(e) \lor \bigvee_j \lnot D_j(s_1) \lor \dots \lor \bigvee_j \lnot D_j(s_k)
  \end{align*}
  Since \[   (\lnot A(s_1) \land \ldots \land \lnot A(s_k))
  \lor {} \lnot \lnot A(s_1) \lor \dots \lor \lnot\lnot A(s_k) 
\] is provable from weak excluded middle, the claim is proved.
\end{proof}

\section{Elimination Sets using $\Bm m$ and $\LIN$}
\label{sec:positive}

We've showed in Theorem~\ref{thm:forking} that the provability of~$\Bm
m$ for some $m \ge 2$ is a necessary condition for an intermediate
logic to have the extended first \et-theorem.  In this section, we
show that it is also sufficient: if $\proves{\Lo} B_m$ for some $m$,
then every critical \et-term has a complete $e$-elimination set.

We use the notion of $e$-elimination sets developed in
Section~\ref{sec:elim-sets}: the existence of an $e$-elimination set
for a set of critical formulas~$\Lambda$ guarantees that these
critical formulas can be removed from a proof, while the end-formula
is replaced by a disjunction of instances of the original end-formula.
The proofs of the existence of $e$-elimination sets proceed by
replacing an \et-term~$e$ in such a way that the disjunction of
formulas~$\Lambda'$ resulting from such replacements become (provable
from) tautologies in the underlying propositional logic.  In
Proposition~\ref{prop:elimCLsingle}, we showed how to do this for a
single critical formula $A(s) \lif A(e)$ in classical logic: We
replace $e$ first by~$s$ resulting in $A(s[s/e]) \lif A(s) \vdash
D(s)$, and then by itself (i.e., no replacement), resulting in $A(s)
\lif A(e) \vdash D(e)$. This gives 
\[
  (A(s[s/e]) \lif A(s)) \lor (A(s) \lif A(e)) \vdash D(s) \lor D(e),
\]
but the formula on the left is a classical tautology. When eliminating
multiple critical formulas at the same time (e.g., all critical
formulas belonging to a single \et-term), the resulting tautologies
are more complicated.  In the original proof, they are all equivalent
to excluded middle, and so the proofs do not apply to intermediate
propositional logics. Below, we show how this can nevertheless be done
as long as the underlying propositional logic contains~$\Bm m$. We
have to distinguish two kinds of critical formulas:
\begin{defn}
  A critical formula $A(s) \lif A(e)$ (or $A(e) \lif A(s)$ if $e$ is a
  $\tau$-term) is called \emph{predicative} if $e$ does not occur
  in~$s$, and \emph{impredicative} otherwise.\footnote{The terminology
  is chosen in analogy to the notion of predicative definition, in
  which the definiens does not itself involve (quantification over)
  the thing being defined. Likewise here, the ``definition'' $A(s)$ of
  the \eps-term~$e$ does not mention the \eps-term~$e$ it defines if
  the critical formula is predicative. Such restrictions of
  definitions (and instances of comprehension) are the basis of
  predicative mathematics, which goes back to Weyl and Russell. There
  is, however, no deeper connection between our choice of terminology
  and predicative mathematics.}
  \end{defn}
If all critical
formulas are predicative, $\Lo \vdash \LIN$ suffices
(Lemma~\ref{lem:pred-lin}).  If $\Lo \vdash \Bm m$ for some~$m$, we
can eliminate the impredicative critical formulas for some
\et-term~$e$ in a similar way: by successively replacing~$e$ by
suitable terms, we obtain a proof of a disjunction of instances
of~$D(e)$ from a formula provable from~$\Bm m$. The number~$m$
determines the number of necessary replacements. Once impredicative
critical formulas are removed, we can use Lemma~\ref{lem:pred-lin} to
remove the remaining predicative critical formulas.

We prove the result for impredicative critical formulas first
(Lemma~\ref{lem:impred-Bn}). In preparation for the proof we first
consider an example to illustrate the basic idea. Suppose
$\proves{\etl{\LC_3}}
D(e)$ with the set of critical formulas
\begin{align*}
  C_s(e) & \equiv A(s(e)) \lif A(e) \\
  C_t(e) & \equiv A(t(e)) \lif A(e) \\
  C_u(e) & \equiv A(u) \lif A(e) \\
  C_v(e) & \equiv A(v) \lif A(e) 
\end{align*}
where terms $u$ and $v$ do not contain~$e$. Let $C(e)$ be the
conjunction of these critical formulas. We have $C_s(e), C_t(e),
C_u(e), C_v(e) \vdash_{\LC_3} D(e)$.

We'll consider sets of
terms~$X_i$ where $X_0 = \{e\}$ and $X_{i+1} = \{ s(x), t(x) : x \in
X_i\}$. For the sake of readability we will leave out parentheses then
writing these terms, e.g., $s(t(e))$ is abbreviated as $ste$.

For every $w \in X_1$ we have $C(w) \vdash D(w)$. So, by applying
Lemma~\ref{lem:facts}(\ref{disj}) twice,  we get:
\begin{align*}
  C(se) \lor C_s(e),&\\ 
  C(te) \lor C_t(e), &\\
  C_u(e), C_v(e) & \vdash_{\LC_3} D(e) \lor D(se) \lor D(te)
\intertext{By distributivity, $C(se) \lor C_s(e)$ is equivalent to}
  (C_s(se) \lor C_s(e)) & \land 
  (C_t(se) \lor C_s(e)) \land {}\\
 (C_u(se) \lor C_s(e)) & \land 
(C_v(se) \lor C_s(e)) 
\intertext{Thus we get}
  C_s(se) \lor C_s(e),
  C_t(se) \lor C_s(e),& \\
  C_u(se) \lor C_s(e),
  C_v(se) \lor C_s(e),& \\
  C(te) \lor C_t(e),&\\
   C_u(e), C_v(e) &\vdash_{\LC_3} D(e)
  \lor D(se) \lor D(te).
\end{align*}
We have
\[
  C_u(e) \equiv A(u) \lif A(e) \vdash_{\LC_3} (A(u) \lif A(se)) \lor (A(se) \lif A(e)) \equiv C_u(se) \lor C_s(e)
\]
using \LIN. Similarly, $C_v(e) \vdash_{\LC_3} C_v(se) \lor C_s(e)$, so we
have by Lemma~\ref{lem:facts}(\ref{impl}):
\begin{align*}
  C_s(se) \lor C_s(e),
  C_t(se) \lor C_s(e),& \\
  C(te) \lor C_t(e),&\\
  C_u(e), C_v(e) & \vdash_{\LC_3} D(e) \lor D(se) \lor D(te)
\intertext{Repeating this consideration with $C(te) \lor C_t(e)$ yields}
  C_s(se) \lor C_s(e),
  C_t(se) \lor C_s(e),& \\
  C_s(te) \lor C_t(e),
  C_t(te) \lor C_t(e),& \\
  C_u(e), C_v(e) &\vdash_{\LC_3} D(e)
  \lor D(se) \lor D(te)
\end{align*}
Note that each of the four resulting disjunctions has as first
disjunct a substitution instance of a critical formula of the form
$C_s(w)$ or $C_t(w)$ where $w \in X_1$. $X_2$ are the terms of the
form $s(w)$ and $t(w)$.  So we can repeat the process, pairing
$C(s(w))$ with $C_s(w)$ and $C(t(w))$ with $C_t(w)$, i.e., obtaining $C(sse)
\lor C_s(se) \lor C_s(e)$, $C(tse) \lor C_t(se) \lor C_s(e)$,
etc. In each case, after distributing and removing conjuncts of the
form $C_u(w) \lor \dots$ we are left with now eight disjunctions:
\begin{align*}
  & C_s(sse) \lor C_s(se) \lor C_s(e),\\
  & C_t(sse) \lor C_s(se) \lor C_s(e),\\
  & \vdots\\
  & C_t(tte) \lor C_t(te) \lor C_t(e),\\
  & C_u(e), C_v(e) \vdash_{\LC_3} D(e)
  \lor D(se) \lor D(te) \lor D(sse) \lor \dots \lor D(tte)
\end{align*}
It remains to show that the formulas on the right of the turnstile are
provable in~$\LC_3$. First, consider a formula of the form
$C_i(w) \lor \dots \lor C_j(e)$, e.g.,
\begin{align*}
  & C_s(tse) \lor C_t(se) \lor C_s(e), \text{ i.e.,}\\
  & (A(stse)) \lif A(tse)) \lor (A(tse) \lif A(se)) \lor (A(se) \lif A(e))
\intertext{In each disjunct, the consequent equals the antecedent of the disjunct
immediately to the right, i.e., it is a substitution instance of}
  & (A_1 \lif A_2) \lor (A_2 \lif A_3) \lor (A_3 \lif A_4)
\end{align*}
i.e., of $\Bm 3$. Since $\LC_3 \vdash \Bm 3$, these are all provable.

If we take $D'(e)$ to be the disjunction obtained on the right, we have
\begin{align*}
  C_u(e), C_v(e) & \vdash_{\LC_3} D'(e) \text{ and thus also}\\
  C_u(u), C_v(u) & \vdash_{\LC_3} D'(u) \text{ and}\\
  C_u(v), C_v(u) & \vdash_{\LC_3} D'(v).
\intertext{As $C_u(u)$ and $C_v(v)$ are of the form $A \lif A$ this reduces to}
C_u(v) & \vdash_{\LC_3} D'(u) \text{ and}\\
C_v(u) & \vdash_{\LC_3} D'(v) \text{ and therefore:}\\
C_u(v) \lor C_v(u) & \vdash_{\LC_3} D'(u) \lor D'(v).
\end{align*}
But $C_u(v) \lor C_v(u)$ is an instance of \LIN.

\begin{lem}\label{lem:iterated-lin}
  If $\pl{\Lo} \vdash \LIN$, then for all $m$, 
  \[A_1 \lif A_{m+1} \vdash_\Lo (A_1 \lif A_2) \lor \dots \lor (A_m \lif A_{m+1}).\]
\end{lem}

\begin{proof}
  By induction on $m$. If $m=1$, this amounts to the claim: $A_1 \lif
  A_2 \vdash_\Lo A_1 \lif A_2$, which is trivial. Now suppose the claim
  holds for $m$. Then
  \begin{align*}
    A_1 \lif A_{m+2} & \vdash_\Lo (A_1 \lif A_{m+1}) \lor (A_{m+1} \lif A_{m+2})
  \intertext{from the instance $(A_1 \lif A_{m+1}) \lor (A_{m+1} \lif A_1)$ of \LIN. 
  By induction hypothesis,}
  A_1 \lif A_{m+1} & \vdash_\Lo (A_1 \lif A_2) \lor \dots \lor (A_m \lif A_{m+1})
  \end{align*}
  and the claim follows by \IL.
\end{proof}

\begin{lem}\label{lem:impred-Bn}
Suppose $\proves{\Lo} B_m$. If $\Delta(e)$ are the impredicative critical
formulas belonging to the $\et$-term~$e$, then $\Delta(e)$ has an
$e$-elimination set.
\end{lem}

\begin{proof}
  Suppose $\Gamma, \Delta(e), \Pi(e) \vdash_{\Lo} D(e)$, where
  $\Pi(e)$ is the set of predicative critical formulas belonging to~$e$.
  
  Suppose $\Pi(e)$ and $\Delta(e)$ consist of, respectively, the
  critical formulas
  \begin{align*}
    C_{u_i}(e) & \equiv A(u_i) \lif A(e)\\
    C_{s_i}(e) & \equiv A(s_i(e)) \lif A(e).
  \end{align*}
  The proof generalizes the preceding example: we successively
  substitute terms for~$e$ in such a way that a disjunction of
  instances of $D$ is implied by substitution instances of the
  critical formulas in $\Gamma$ together with a disjunction of the
  form~$\Bm m$ plus the predicative critical formulas $\Pi(e)$. Once $k
  = m$, the disjunction becomes provable from $\Bm m$.

  Let $T_0 = \{e\}$ and $T_{i+1} = \{s_j(t) : t \in T_i, j \le r\}$.
  Let $\Gamma(T) = \{C[t/e] : C \in \Gamma, t \in T\}$.
  If $w$ is a word over $\{s_1, \dots, s_r\}$, i.e., $w = s_{i_1}\dots
  s_{i_k}$ then we write $w_j(t)$ for $s_{i_j}(s_{i_{j+1}}(\dots
  s_{i_k}(t)\dots))$. So e.g., if $w=sst$ then $w_1(e)$ is
  $s(s(t(e)))$, $w_3(e) = t(e)$ and $w_4(e) = e$. Let $W(m)$ be the
  set of all length~$m$ words over $s_1$, \dots, $s_r$. 

  We show by induction on $m$ that
  \begin{align*}
    \Gamma, \Gamma(T_1), \dots, \Gamma(T_{m-1}), \Lambda_m, \Pi(e) & 
    \vdash_\Lo \bigvee_{i=1}^m \bigvee_{t\in T_{i-1}} D(t)
  \intertext{where}
    \Lambda_i & = \{ \bigvee_{i=1}^m C(w,i) : w \in W(m)\} \text{ and}\\
    C(w,i) & \equiv A(w_i(e)) \lif A(w_{i+1}(e)).
  \intertext{The induction basis is $m=1$. Then}
    T_1 & = \{s_1(e), \dots, s_r(e)\},\\
    W(1) & = \{s_1, \dots, s_r\},\\
    C(s_j,1) & \equiv A(s_j(e)) \lif A(e)
  \end{align*}
  and so each disjunction in $\Lambda_i$ is just one of the
  impredicative critical formulas $A(s_j(e)) \lif A(e))$, i.e.,
  $\Lambda_1 = \Delta(e)$. Likewise, the disjunction on the right is
  just~$D(e)$. So the claim holds by the assumption that $\Gamma,
  \Delta(e), \Pi(e) \vdash D(e)$.

  Now let $v = s_{i_1}\dots s_{i_{m}}$ be a length~$m$ word, and
  \begin{align*}
    \Lambda_m(v) & = \{ \bigvee_{i=1}^m (A(w_i(e)) \lif A(w_{i+1}(e)) : 
    w \in W(m) \setminus \{v\}\}\\
    C(v) & \equiv \bigvee_{i=1}^m (A(v_i(e)) \lif A(v_{i+1}(e))
  \end{align*}
  In other words, $\Lambda_m = \Lambda_m(v) \cup \{C(v)\}$. 
  We'll abbreviate $\Gamma(T_1)$, \dots,~$\Gamma(T_{m-1})$ as~$\Gamma'$,
  and $\bigvee_{i=1}^m \bigvee_{t\in T_{i-1}} D(t)$ as~$D'$. The
  induction hypothesis can then be written as:
  \begin{align*}
    \Gamma, \Gamma', C(v), \Lambda_m(v), \Pi(e) & 
    \vdash_\Lo D'
  \intertext{Take $t = v_1(e)$ i.e., $s_{i_1}(\dots s_{i_{m}}(e))$. By replacing $e$ 
  by~$t$ in~$\pi$, we have}
      \Gamma(t), \Lambda(t), \Pi(t) & \vdash_\Lo D(t) \text{ and so also}\\
    \Gamma(t), \bigwedge \Lambda(t) \land \bigwedge \Pi(t) & \vdash_\Lo D(t)
  \intertext{Combining this with the induction hypothesis using 
  Lemma~\ref{lem:facts}(\ref{disj}) we have}
  \Gamma, \Gamma', \Gamma(t), &\\
   (\bigwedge \Lambda(t) \land \bigwedge \Pi(t)) \lor C(v), &\\
    \Lambda_m(v), \Pi(e) & 
    \vdash_\Lo D' \lor D(t)
  \intertext{If we write $\Xi \lor G$ for $\{F \lor G : F \in \Xi\}$, by distributivity,}
  \Gamma, \Gamma', \Gamma(t),&\\
   \Lambda(t) \lor C(v),&\\
    \Pi(t) \lor C(v),&\\
     \Lambda_m(v), \Pi(e) & 
  \vdash_\Lo D' \lor D(t)
  \end{align*}
  Recall that $t \equiv v_1(e)$ where $v$ is a word of length~$m$. 
  The formulas in $\Pi(t)$ are of the form $A(u_i) \lif A(v_1(e))$,
  so a formula in $\Pi(t) \lor C(v)$ is of the form
  \[
    (A(u_i) \lif A(v_1(e))) \lor 
    (A(v_1(e)) \lif A(v_2(e))) \lor \dots \lor 
    (A(v_m(e)) \lif A(e)).\]
  Every such formula is implied by $A(u_i) \lif A(e)$ by
  Lemma~\ref{lem:iterated-lin}, since $\IL + \Bm m \vdash \LIN$ by
  Proposition~\ref{prop:Bn-lin}. Since $A(u_i) \lif A(e)$ is in~$\Pi(e)$, 
  we get:
  \[
  \Gamma, \Gamma', \Gamma(t), \Lambda(t) \lor C(v), \Lambda_m(v), \Pi(e)  
  \vdash_\Lo D' \lor D(t)
  \]
  Every formula in $\Lambda(t) \lor C(v)$ is of the form~$\Bm m$, specifically,
  \[
    (A(s_i(v_1(e))) \lif A(v_1(e))) \lor 
    (A(v_1(e)) \lif A(v_2(e))) \lor \dots \lor 
    (A(v_m(e)) \lif A(e)).
  \]
  Since for each $i \le r$, $A(s_i(v_1(e))) \lif A(v_1(e)) \in
  \Lambda(t)$, $\Lambda(t) \lor C(v)$ is the set of all disjunctions
  \[
    \bigvee_{i=1}^{m+1} (A(w_i(e)) \lif A(w_{i+1}(e))
  \]
  where $w = s_iv$ for some $i \le r$. As every length $m+1$ word is
  of this form for some length~$m$ word~$v$, repeating this process for all
  length~$m$ words~$v$ thus yields
  \begin{align*}
  \Gamma, \Gamma', \Gamma(T_m), \Lambda_{m+1}, \Pi(e) & 
  \vdash_\Lo D' \lor \bigvee_{t \in T_m} D(t) 
  \intertext{As we've seen, a formula in $\Lambda_m$ is of the form~$\Bm m$, so in
  $\pl{\Lo} + \Bm m$, we have}
    \Gamma, \Gamma(T_1), \dots, \Gamma(T_{m-1}), \Pi(e) 
    & \vdash_\Lo \bigvee_{i=1}^m \bigvee_{t\in T_{i-1}} D(t) 
  \end{align*}
  Thus, the claim follows by taking $T = \{e\} \cup T_{m-1}$.

  If $e$ is a \tau-term, the proof proceeds analogously. The
  resulting formulas in~$\Lambda(e)$ are then of the form $(A_m \lif
  A_{m+1}) \lor \dots \lor (A_1 \lif A_2)$ which is equivalent to~$\Bm
  m$.
\end{proof}

\begin{lem}\label{lem:bigdisj}
  If $\proves{\Lo} \LIN$,
  \begin{enumerate}
  \item $\proves{\pl{\Lo}} \bigvee_{j=1}^m \bigwedge_{i=1}^m (A_i \lif A_j)$
  \item $\proves{\pl{\Lo}} \bigvee_{j=1}^m \bigwedge_{i=1}^m (A_j \lif A_i)$
  \end{enumerate}
\end{lem}

\begin{proof}
  As propositional infinite-valued G\"odel logic~$\G_\Real$ is
  axiomatized by $\IL + \LIN$, it suffices to show that the formulas
  are valid in the G\"odel logic based on the truth value set $[0,1]$.
  For (1), in any given valuation, one of the~$A_j$ must be maximal,
  i.e., $A_i \lif A_j$ has value~$1$ in it for all~$i$. For~(2), one
  of the~$A_j$ must be minimal.
\end{proof}

\begin{lem}\label{lem:pred-lin} Suppose $\proves{\Lo} \LIN$ and
  $\Gamma, \Pi(e) \vdash_{\pl{\Lo}} D(e)$, where $\Pi(e)$ are the
  predicative critical formulas $A(u_j) \lif A(e)$ belonging to~$e$,
  there are no impredicative critical formulas belonging to~$e$, and
  $\Gamma$ are critical formulas for which $e$ is not critical. Then
  $\Pi(e)$ has a complete $e$-elimination set.
\end{lem}

\begin{proof}
  First suppose $e$ is an \eps-term. Replacing $e$ by~$u_j$ results in
  a proof showing
  \begin{align*}
    \Gamma[u_j/e], \bigwedge_{i=1}^p (A(u_i) \lif A(u_j)) & \vdash_\Lo D(u_j) 
  \intertext{By applying Lemma~\ref{lem:facts}(\ref{disj}), 
  we get}
    \bigcup_j \Gamma[u_j/e], \bigvee_{j=1}^p \bigwedge_{i=1}^p (A(u_i) \lif A(u_j)) 
    & \vdash_\Lo \bigvee_{j} D(u_j) 
  \end{align*}
  The disjunction of conjunctions on the left is provable in $\pl{\Lo} +
  \LIN$ by Lemma~\ref{lem:bigdisj}(1).

If $e$ is a \tau-term, we get
  \[
    \bigcup_j \Gamma[u_j/e], \bigvee_{j=1}^p \bigwedge_{i=1}^p (A(u_j) \lif A(u_i)) 
  \vdash_\Lo \bigvee_{j} D(u_j) 
\]
and the claim follows by Lemma~\ref{lem:bigdisj}(2).
\end{proof}

\begin{thm}\label{thm:elim-Gm}
  If $\pl{\Lo} \vdash \Bm m$ for some $m$, then $\etl{\pl{\Lo}}$ has
  complete $e$-elimination sets.
\end{thm}

\begin{proof}
  Suppose $\Gamma, \Pi(e), \Delta(e) \vdash_{\pl{\Lo}} D(e)$, where
  $\Pi(e)$ are the predicative critical formulas belonging to~$e$,
  $\Delta(e)$ the impredicative formulas belonging to $e$, and
  $\Gamma$ are critical formulas for which $e$ is not critical. By
  Lemma~\ref{lem:impred-Bn}, $\Gamma[T], \Pi(e) \vdash_{\pl{\Lo}}
  \bigvee_{t\in T} D(t)$ where $T = \{e\} \cup T_{m-1}$. Since
  $\proves{\pl{\Lo}+\Bm m} \LIN$ by Proposition~\ref{prop:Bn-lin},
  Lemma~\ref{lem:pred-lin} applies and so $\bigcup_{j} T[u_j/e]$ is a
  complete $e$-elimination set.
\end{proof}

\begin{rem}\label{alternateC} The proof of
  Proposition~\ref{prop:elimCL} provides essentially Hilbert's way of
  computing $e$-elimination sets using excluded middle. However,
  instead of excluded middle $A \lor \lnot A$, classical logic can
  also be axiomatized over~\IL{} by $\Bm 2$, i.e., $(A \lif B) \lor (B
  \lif C)$. The method of computing $e$-elimination sets using Lemmas
  \ref{lem:impred-Bn} and~\ref{lem:pred-lin} applied to $\LC_2 = \CL$
  provides a method for computing $e$-elimination sets (and hence of
  Herbrand disjunctions) different from Hilbert's method.
\end{rem}

\section{The Hilbert-Bernays Elimination Procedure}
\label{sec:hb-elim}

Recall that the challenges in the proof of the extended first
\et-theorem include, in addition to the existence of complete
elimination sets, guarantees that the new sets $\Gamma[s_i/e]$ are in
fact critical formulas (so eliminating a set $\Lambda(e)$ of critical
formulas yields a correct $\etl{\Lo}$-proof), and that the process
eventually terminates. In Hilbert and Bernays's original proof of the
first \eps-theorem, this was ensured by processing sets of critical
formulas in a specific order.  We briefly review this proof,
concentrating on its structure, since we'll apply the same method to
critical formulas in \et-proofs for intermediate logics.

\begin{defn}
  Suppose $\vdash_{\etl{\Lo}}^{\pi_1} D_1$. A sequence
  $\langle\Lambda_1(e_1),T_1\rangle$, \dots,
  $\langle\Lambda_k(e_k),T_k\rangle$ is an \emph{\et-elimination
  sequence} iff, for each $i$,
  \begin{enumerate}
    \item $\Lambda_i(e_i)$ is a set of critical formulas belonging
    to~$e_i$,
    \item $\Gamma_i, \Lambda'(e), \Lambda_i(e_i)
    \vdash_{\etl{\Lo}}^{\pi_i} D_i$, where $\Lambda'(e)$ are the
    critical formulas for~$e_i$ not in $\Lambda_i(e_i)$, and $\Gamma_i$
    the remaining critical formulas in~$\pi_i$,
    \item $T_i$ is an $e_i$-elimination set for $\pi_i$ and $\Lambda_i(e_i)$
    \item $\Gamma_{i+1} = \Gamma_i[T_i/e_i]$ and $D_{i+1} =
    \bigvee_{t\in T_i} D_i[t/e_i]$, 
    \item $\Gamma_{i+1} \vdash_{\etl{\Lo}} D_{i+1}$,
  \end{enumerate}
  and $\Gamma_{k+1} = \emptyset$, that is, $\proves{\Lo}^{\pi_{k+1}}
  D_{k+1}$. If all $\Lambda'(e_i) = \emptyset$ (i.e., $T_i$ is a
  complete elimination set for~$e_i$), we say the sequence is a
  complete elimination sequence.
\end{defn}

An $\et$-elimination sequence is a sequence of \et-terms~$e_i$ and
sets of critical formulas~$\Lambda(e_i)$ belonging to it, such that
eliminating $\Lambda_i(e_i)$ from proof~$\pi_i$ results in a new proof
$\pi_{i+1}$ of a disjunction of instances of~$D(e_i)$. This new proof
proceeds from instances of the critical formulas for which $e_i$ is
not critical and the remaining critical formulas belonging to~$e_i$.
Since the definition requires $\pi_i$ to be an \et-proof, the formulas
in $\Gamma[T_i]$ must actually be critical formulas.

If an \et-elimination sequence exists for a formula~$E$ and its
\et-proof~$\pi_0$, then the extended first \et-theorem holds for~$E$:

\begin{prop}\label{prop:first-et}
  Suppose $\pi$ is an $\etl{\pl{\Lo}}$ proof of $E(u_1, \dots, u_n)$
  where $E(x_1, \dots, x_n)$ is \et-free. Suppose furthermore that an
  \et-elimination sequence exists for~$\pi$. Then there are tuples of
  terms $t_{i1}$, \dots, $t_{in}$ such that $\proves{\Lo} \bigvee_{i=1}^l
  E(t_{i1}, \dots, t_{in})$.
\end{prop}

\begin{proof}
  Since $E(x_1, \dots, x_n)$ is \et-free, $E(t_{i1}, \dots,
  t_{in})[s/e] \equiv E(t_{i1}[s/e], \dots, t_{in}[s/e])$. The result
  follows by induction on~$k$, the length of the \et-elimination
  sequence for~$\pi$.
\end{proof}

For the proof of the extended first \et-theorem, then, it is
sufficient to show that suitable $e$-elimination sets always exist,
and that \et-terms~$e$ and sets of associated critical formulas can be
successively chosen in such a way as to yield an \et-elimination
sequence for~$\pi_1$. Hilbert and Bernays did this by defining a
well-ordering of \et-terms with the property that eliminating maximal
\et-terms according to this ordering guarantees that the $\Gamma[T_i]$
are again critical formulas, and that in each step no critical
\et-terms are newly introduced which are larger (in the ordering). 

The ordering used is the lexicographic order on two complexity
measures of \et-terms~$e$. We say that $e$ is \emph{nested in} an
\et-term~$e'$ if $e$ is a proper subterm of~$e$, i.e., if every
occurrence of a variable which is free in~$e$ is also free in~$e'$. An
\et-term $e$ is \emph{subordinate} to $e' \equiv \meps x A(x)$ or
$\equiv \mtau x A(x)$ if $e$ occurs in $e'$ and $x$ is free in~$e$.
The \emph{degree}~$\deg e$ of $e$ is the maximal level of nesting of
subterms of~$e$; the rank~$\rk e$ the maximal level of subordination.
When $C$ is a critical formula belonging to~$e$, we let $\deg C = \deg
e$ and $\rk C = \rk e$. 

The Hilbert-Bernays elimination order proceeds by always picking a
critical formula of maximal degree among the critical formulas of
maximal rank.  Its success relies on the following two lemmas, which
establish that (a)~replacement of maximal $\et$-terms in a critical
formula results in a critical formula, and (b)~if new critical
formulas are generated, they are of lower rank, or of the same rank
but of lower degree.  

\begin{lem}\label{lem:rep-eps-crit} Suppose $e$ is an \et-term,
$t$~any term, $C$~is a critical formula for which $e$ is not critical,
and $\rk C \le \rk e$. Then $C[t/e]$ is also a critical formula.
\end{lem}

\begin{lem}\label{lem:rep-eps-lower} Suppose $e$ is an \et-term,
  $t$~any term, $C$~is a critical formula for which $e$ is not
  critical, $\rk C \le \rk e$, and if $\rk C = \rk e$ then $\deg C \le
  \deg e$. Then $\rk{C[t/e]} \le \rk e$, and if $\rk{C[t/e]} = \rk e$,
  then $\deg{C[t/e]} \le \deg e$
\end{lem}

The proofs of the preceding lemmas are as in
\cite{HilbertBernays1939}; see also \cite[\S5]{MoserZach2006} and
\cite[\S4.1]{Zach2017}.

\begin{prop}\label{prop:term-elim-seq}
  Suppose $\pi$ is an $\etl{\pl{\Lo}}$ proof of $E(u_1, \dots, u_n)$
  where $E(x_1, \dots, x_n)$ is \et-free. Suppose furthermore that for
  every $\etl{\pl{\Lo}}$ proof and critical \et-term~$e$ there is a
  complete $e$-elimination set. Then there is a complete \et-elimination
  sequence for~$\pi$.
\end{prop}

\begin{proof}
  Take $\pi_0 = \pi$. Suppose $\pi_i$ has been defined. Let $e_i$ be a
  critical \et-term of $\pi_i$ of maximal degree among the critical
  terms of maximal rank. Let $\Lambda_i(e_i)$ be all critical
  formulas belonging to~$e_i$, and $\Gamma_i$ the remaining critical
  formulas. We have 
  \begin{align*}
    \Gamma_i, \Lambda_i(e_i) & \vdash_{\etl{\Lo}} D_i
\intertext{By assumption, there is a complete $e$-elimination set $T_i$
for~$\Lambda_i(e_i)$ and so we have $\pi_{i+1}$ showing that}
    \Gamma_i[T_i] & \vdash_{\etl{\Lo}} \bigvee_{t\in T_i} D_i[t/e_i]
  \end{align*}
  Each critical formula~$C$ in $\Gamma_i$ is not of higher rank
  than~$e_i$, and if it is of equal rank it is not of higher degree.
  So by Lemma~\ref{lem:rep-eps-crit}, $C[t/e_i]$ is a critical
  formula. Hence $\pi_{i+1}$ is a correct $\etl{\Lo}$-proof. Let
  $\Gamma_{i+1} = \Gamma_i[T_i]$ and $D_{i+1} \equiv \bigvee_{t\in
  T_i} D_i[t/e_i]$.

  Eventually, $\Gamma_i = \emptyset$, since in each step, by
  Lemma~\ref{lem:rep-eps-lower}, the maximal rank of critical
  \et-terms in $\Gamma_i$ does not increase, the maximal degree of
  critical \et-terms of maximal rank does not increase, and the number
  of critical \et-terms of maximal degree among those of maximal rank
  decreases.
\end{proof}

\begin{cor}\label{first-eps-cl}
  The extended first \et-theorem holds for~$\CL$.
\end{cor}

\begin{proof}
  By Proposition~\ref{prop:elimCL}, every critical \et-term has
  complete elimination sets. So by
  Proposition~\ref{prop:term-elim-seq}, there always is an elimination
  sequence. The extended first \et-theorem follows by
  Proposition~\ref{prop:first-et}.
\end{proof}

\begin{rem}\label{rem:speedup} The traditional procedure following the
  Hilbert-Bernays order, which eliminates all critical formulas
  belonging to a maximal \et-term together, is not the only possible
  procedure that guarantees termination. We pointed out in
  Remark~\ref{rem:one-by-one} that using
  Proposition~\ref{prop:elimCLsingle} for all critical formulas
  belonging to a single \et-term results in a larger disjunction than
  Proposition~\ref{prop:elimCL}. Despite this, the ability in the
  classical case to eliminate single critical formulas provides
  flexibility that can be exploited to produce smaller overall
  Herbrand disjunctions. As \citet[Theorem~3]{BaazLeitschLolic2018}
  show, there are sequences of \et-proofs where the original procedure
  produces Herbrand disjunctions that are non-elementarily larger than
  a more efficient elimination order.
\end{rem}

\begin{thm}\label{et-neg-KC} The extended first \et-theorem holds for
  negated formulas in any~$\etl{\Lo}$ such that $\proves{\Lo} \J$,
  e.g., $\etl\KC$ and~$\etl\LC$.
\end{thm}

\begin{proof}
  $\Lo$ has complete $e$-elimination sets for end-formulas that are
  disjunctions of negated formulas, by Proposition~\ref{prop:elimJ}.
  Note that~$\J$, i.e., $\lnot A \lor \lnot\lnot A$, follows
  intuitionistically from $(A \lif \lnot A) \lor (\lnot A \lif A)$,
  which is an instance of~$\LIN$. So $\proves{\LC} \J$.  
\end{proof}

\begin{thm}\label{thm:et-Gm}
  The extended first \et-theorem holds for~$\etl{\LC_m}$.
\end{thm}

\begin{proof}
  By Theorem~\ref{thm:elim-Gm}, every critical \et-term has complete
  elimination sets. The extended first \et-theorem follows by
  Propositions~\ref{prop:term-elim-seq} and~\ref{prop:first-et}.
\end{proof}

\begin{thm}\label{thm:weak-eps-LC} The first \et-theorem holds in
  $\etl\LC$, i.e., if $\proves{\etl\LC} D$ for an $\et$-free
  formula~$D$, then $\proves{\LC} D$.
\end{thm}

\begin{proof}
  If $\proves{\etl{\LC}} D$ then also $\proves{\etl{\LC_m}} D$. Since
  $D$~is \et-free, $\proves{\LC_m} D$ by Theorem~\ref{thm:et-Gm}. In
  general $\proves{\LC} D$ iff $\proves{\LC_m} D$ for all~$m$, so the
  claim follows.
\end{proof}

We are indebted to the referee for the \textsc{Journal} for the
following observation:

\begin{prop}\label{neg-IL-first-eps}
  The first \et-theorem holds for negated formulas in $\etl{\Lo}$ for
  any intermediate logic~$\Lo$, including~$\IL$:
  if $\proves{\etl\Lo} \lnot D$ for an $\et$-free
  formula~$D$, then $\proves{\Lo} D$.
\end{prop}

\begin{proof}
  If $\proves{\etl\Lo} \lnot D$ then also $\proves{\etl\CL} \lnot D$.
  By the first \eps-theorem for~$\CL$ (Corollary~\ref{first-eps-cl}),
  $\proves{\CL} \lnot D$, and by Glivenko's Theorem, $\proves{\IL}
  \lnot D$.
\end{proof}

\section{Herbrand's Theorem and the Second Epsilon-Theorem}
\label{sec:herbrand}

The extended first \et-theorem implies Herbrand's theorem for purely
existential formulas. If $E \equiv \exists x_1\ldots \exists x_n
E'(x_1, \ldots, x_n)$ is provable in predicate logic, then so is
$E^\et \equiv E'(e_1, \ldots, e_n)$ for some \eps-terms $e_1$,
\dots,~$e_n$. From the extended first \eps-theorem we then obtain a
proof in propositional logic of a Herbrand disjunction
\[
E'(t_{1}^1, \ldots, t_{n}^1) \lor \ldots \lor A(t_{1}^k, \ldots, t_{n}^k)
\]
for some terms~$t_{i}^j$. In predicate logic, we may now successively
introduce existential quantifiers to obtain the original formula~$E$.
This holds in any intermediate predicate logic in which the first
\et-theorem holds, since the only principles used in the last step
(proving $E$ from its Herbrand disjunction) are $A(t) \lif \exists x
\,A(x)$ and $\exists x(A(x) \lor B) \lif (\exists x\,A(x) \lor B)$,
($x$ not free in~$B$) which already hold in intuitionistic logic.
(Despite the failure of the extended first \et-theorem in $\etl{\LC}$, the
Herbrand theorem for existential formulas does hold in $\ql\LC$; see
\citealt{Aschieri2017}).

For classical predicate logic, the Herbrand theorem for existential
formulas implies the Herbrand theorem for prenex formulas. If a prenex
formula~$E$ has a proof in first-order logic, so does its (purely
existential) \emph{Herbrand form} $H(E) = \exists x_1\ldots \exists
x_n E'(x_1, \ldots, x_n)$. In classical predicate logic, we can obtain
not just $H(E)$ from its Herbrand disjunction, but also the original
prenex formula~$E$. This requires that we introduce not just
existential quantifiers, but also universal quantifiers. Consequently,
we need not just the generalization rule $A(x) / \forall x\,A(x)$ but
also a principle that allows us to shift universal quantifiers over
disjunctions, viz.,
\begin{equation*}
  \forall x(A(x) \lor B) \lif (\forall x\,A(x) \lor B). \tag{\CD}
\end{equation*}
This principle is not intuitionistically valid: it characterizes
Kripke frames with constant domains. Since the implication $E \lif
H(E)$ already holds in intuitionistic logic, any intermediate
predicate logic in which~$\CD$ holds and which has the extended first
\eps-theorem also has Herbrand's theorem for prenex formulas:

\begin{prop}
  Suppose ${\ql{\Lo} + \Ax}$ proves~$\CD$ and $\etl{\Lo}$ has the
  extended first \et-theorem. Then $\ql{\Lo} + \Ax$ has Herbrand's
  theorem for prenex formulas.
\end{prop}

Although the extended first \eps-theorem implies Herbrand's theorem
for prenex formulas if $\CD$ is provable, the converse is not true. As
we showed in Theorem~\ref{thm:forking}, the extended first \et-theorem
holds for $\etl\Lo$ only if $\proves{\Lo}\Bm m$ for some~$m$. In
particular, it does not hold for infinite-valued first-order G\"odel
logics~$\G_\Real = \ql{\LC} + \CD$. However, Herbrand's theorem for
prenex formulas does hold for~$\G_\Real$; see
\cite[Theorem~5.7]{BaazPreiningZach2003} and
\cite[Theorem~7.8]{BaazPreiningZach2007}. (Incidentally, the Herbrand
theorem for prenex formulas also holds in intuitionistic logic despite
the invalidity of~$\CD$; see \citealt{Bowen1976}).

In classical logic, the second \eps-theorem can be proved using the
extended first  \eps-theorem as follows: Suppose $\proves{\etl{\CL}}
A^\et$. Since $A$ is equivalent to a prenex formula $A^p$ in classical
logic, we have $\proves{\ql{\CL}} A \lif A^p$. Prenex formulas imply
their Herbrand forms, i.e., $\proves{\ql{\CL}} A^p \lif H(A^p)$.
Together we have $\proves{\ql{\CL}} A \lif H(A^p)$ and by translating
into the \et-calculus, $\proves{\etl{\CL}} A^\et \lif H(A^p)^\et$, so
$\proves{\etl{\CL}} H(A^p)^\et$. By the extended first \et-theorem,
$H(A^p)$ has a Herbrand disjunction, from which (in $\ql{\CL}$) we can
prove $A^p$ and hence~$A$.

The steps that may fail in an intermediate prediate logic~$\ql{\Lo} +
\Ax$, other than the extended first \et-theorem, are the provability
of $A \liff A^p$ and proving $A^p$ from the Herbrand disjunction of
$H(A^p)$. These steps do work provided all quantifier shifts can be
can be carried out (i.e., in addition to $\CD$ also the formulas
$\Aout$ and~$\Eout$). Thus, if $\etl{\Lo}$ has the extended first
\et-theorem, any intermediate predicate logic~$\ql{\Lo} + \Ax$ in which all
quantifier shifts are provable also has the second \et-theorem:

\begin{prop}\label{prop:second} Suppose $\ql{\Lo} + \Ax$ proves $\CD$,
  $\Eout$, and $\Aout$, and $\etl\Lo$ has the extended first
  \et-theorem. Then $\ql{\Lo} + \Ax$ has the second \et-theorem, i.e.,
  if $\proves{\etl{\Lo}} A^\et$ then $\proves{\ql{\Lo}+\Ax} A$.
\end{prop}

\begin{proof}
  Suppose $\proves{\etl{\Lo}} A^\et$. The \et-calculus proves
  \et-translations of all quantifier shifts, so $\proves{\etl{\Lo}}
  (A^p)^\et$, and since $\proves{\ql{\IL}} A^p \lif H(A^p)$ also
  $\proves{\etl{\Lo}} (H(A^p))^\et$. By the extended first \et-theorem
  for~$\etl\Lo$, there is a Herbrand disjunction~$A'$ of~$H(A^p)$ so that
  $\proves{\Lo} A'$. Since $\proves{\ql{\Lo}+\Ax} \CD$,
  $\proves{\ql{\Lo}+\Ax} A^p$. Since $\proves{\ql{\Lo}+\Ax} \Eout$ and
  $\proves{\ql{\Lo}+\Ax} \Aout$, also $\proves{\ql{\Lo}+\Ax} A^p \lif
  A$, and so $\proves{\ql{\Lo}+\Ax} A$.
\end{proof}

Infinite-valued first-order G\"odel logic~$\G_\Real = \ql{\LC} + \CD$
is an intermediate prediate logic which proves~$\CD$ but not~$\Aout$
or~$\Eout$, and the extended first \et-theorem does not hold
for~$\etl\LC$. The logics of linear Kripke frames with $m$~worlds (and
varying domains) are~$\ql{\LC_m}$. But $\ql{\LC_m}$ does not
prove~$\CD$, so it also does not have the second \et-theorem:

\begin{prop}\label{prop:no-second}
  $\G_\Real = \ql{\LC} + \CD$ and $\ql{\LC_m}$ do not have the second \et-theorem.
\end{prop}

\begin{proof}
  We have $\proves{\etl{\LC}}{\Aout}$ but ${\ql{\LC} +
  \CD}\nvdash {\Aout}$, and $\proves{\etl{\LC_m}}{\CD^\et}$ but
  ${\ql{\LC_m}} \nvdash \CD$.
\end{proof}

However, we have:

\begin{cor}\label{second-gm}
  The second \et-theorem holds for finite-valued first-order G\"odel
  logics~$\G_m = \ql{\LC_m} + \CD$. 
\end{cor}

\begin{proof}
  By Proposition~\ref{prop:second}, as $\G_m$ proves $\CD$, $\Eout$,
  and~$\Aout$ and has the extended first \et-theorem.
\end{proof}

Herbrand's theorem also yield other results, for instance the
following.

\begin{prop}
  Suppose $\ql\Lo +\Ax_1 \subseteq \ql\Lo + \Ax_2$ are intermediate
  predicate logics. If Herbrand's theorem holds in $\ql\Lo + \Ax_2$
  for existential formulas~$A$, then $\proves{\ql\Lo+\Ax_1} A$ iff
  $\proves{\ql\Lo+\Ax_2} A$. The result also holds for prenex
  formulas~$A$ if $\proves{\ql\Lo+\Ax_1} \CD$. 
\end{prop}

\begin{proof}
  The ``only if'' direction is trivial since $\ql\Lo+\Ax_1 \subseteq
  \ql\Lo+\Ax_2$. For the ``if'' direction, assume
  $\proves{\ql\Lo+\Ax_2} A$. Then there is a Herbrand disjunction $A'$
  of~$A$ provable in~$\Lo$ and hence in~$\ql\Lo+\Ax_1$. Since
  $\proves{\ql{\IL}} A' \lif A$, also $\proves{\ql\Lo+\Ax_1} A$. 
  
  Now suppose in addition that $\proves{\ql\Lo+\Ax_1} \CD$ and
  $\proves{\ql\Lo+\Ax_2} A$ with $A$ prenex. Then since
  $\proves{\ql{\IL}} A \lif H(A)$, $\proves{\ql\Lo+\Ax_2} H(A)$. By
  Herbrand's theorem there is a Herbrand disjunction~$A'$ and
  $\proves{\Lo} A'$. From $A'$, $\ql\Lo+\Ax_1$ proves~$A$, using just
  intuitionistically valid inferences as well as~$\CD$.
\end{proof}

\begin{cor}
  The existential fragments of $\ql{\LC}$ and~$\G_\Real$ agree.
\end{cor}

\begin{proof}
  Take $\Lo = \LC$, $\Ax_1 = \emptyset$, and $\Ax_2 = \CD$.
  Infinite-valued G\"odel logic $\G_\Real = \ql\LC+\CD$ has Herbrand's
  theorem.
\end{proof}

Whenever the extended first \et-theorem holds for~$\etl{\Lo}$,
Herbrand's theorem for existential formulas also holds for any
intermediate predicate logic containing~$\Lo$. So, whenever an
intermediate predicate logic $\ql{\Lo} +\Ax_1 \subseteq \ql{\Lo}
+\Ax_2$, and $\etl{\Lo}$ has the extended first \et-theorem, the
purely existential fragments of $\ql{\Lo} +\Ax_1$ and $\ql{\Lo}
+\Ax_2$ agree. For instance:

\begin{cor}
The existential fragments of $\ql{\LC_m}$ and $\G_m = \ql{\LC_m} +
\CD$ agree. 
\end{cor}

By similar reasoning, the result holds for formulas of the form
$\exists \vec x\, \lnot A(\vec x)$ for $\KC$ and its extensions, since
$\KC$ has the extended first \et-theorem for negated formulas
(Theorem~\ref{et-neg-KC}).  By
Proposition~\ref{prop:eps-thm-versions}, the result can also be
extended to formulas of the form $\forall \vec y\, B(\vec y) \lif
\exists \vec x\,A(\vec x)$.

\section{Elimination of Critical Formulas using \LIN}
\label{sec:LIN} 

In the classical case, the first \eps-theorem is obtained by
successively eliminating critical formulas belonging to a single
\eps-term using excluded middle. In intermediate logics this is not
available, but as we have seen in the proof of
Theorem~\ref{thm:elim-Gm}, critical formulas can also be eliminated
using $\Bm m$ and \LIN. And in fact, if a critical \et-term $e$ has
only predicative critical formulas then it has a complete
$e$-elimination set (by Lemma~\ref{lem:pred-lin}) already in $\LC$,
since only \LIN{} is required to eliminate predicative critical
formulas. Thus, the procedure of the first \et-theorem terminates for
all \et-proofs in which no impredicative critical formulas occurs
during the successive elimination of critical formulas.

\begin{prop}
  If $\proves{\etl{\LC}} E(u_1,\dots,u_n)$ and there is an elimination
  sequence $\langle \pi_i, \Lambda_i(e_i), T_i\rangle$ in which each formula
  in $\Lambda_i(e_i)$ is predicative, then $\proves{\LC} \bigvee
  E(t_{i1}, \dots, t_{in})$.
\end{prop}

\begin{proof}
  By Proposition~\ref{prop:first-et}, since if in each step of the
  elimination sequence the eliminated critical formulas are all
  predicative, the elimination already works in~\LC{} by
  Lemma~\ref{lem:pred-lin}.
\end{proof}

So if there is a way to select critical \et-terms~$e_i$ successively
for elimination in such way that the critical formulas belonging
to~$e_i$ are always predicative, the extended first \et-theorem holds in
$\etl{\LC}$ for a particular proof~$\pi$. However, it is hard to
determine just by inspecting~$\pi$ if this is possible. For one, it is
not sufficient that the critical formulas in $\pi$ itself are all
predicative: eliminating the critical formulas belonging to one
critical \et-term may turn a remaining predicative critical formula
into an impredicative one. For instance, consider
\begin{align*}
  E(x, y) & \equiv (A(f(y)) \lif A(x)) \land (B(g(x)) \lif B(y))\\
  e_A & \equiv \meps x A(x)\\
  e_B & \equiv \meps y B(y)
  \intertext{Then $D(e_A, e_B)$ has an $\etl{\LC}$ proof, since 
  it is the conjunction of the critical formulas}
  & A(f(e_B)) \lif A(e_A)\\
  & B(g(e_A)) \lif B(e_B)
  \intertext{which are both predicative. If we first eliminate $e_A$ we would replace $e_A$ by $f(e_B)$ in the second, resulting in}
  & B(g(f(e_B)) \lif B(e_B)
  \intertext{which is impredicative. Similarly, eliminating $e_B$ leaves the impredicative}
  & A(f(g(e_A)) \lif A(e_A).
\end{align*}
So no elimination sequence resulting in only predicative critical
formulas at every step is possible.

Of course, if the term $t$ in a critical formula $A(t) \lif A(e)$
contains no \et-term at all, it is predicative, and replacing some
\et-term $e'$ in it by a term~$t'$ cannot result in an impredicative
critical formula. Let us call such critical formulas~\emph{weak}.

\begin{defn}
  A critical formula $A(t) \lif A(\meps x A(x))$ resp.\ $A(\mtau
    xA(x)) \lif A(t)$ is \emph{weak} in $\pi$ if $t$ does
  not contain any critical $\eps$- or $\tau$-term of~$\pi$.
\end{defn}

If the critical formulas in $\pi$ are all weak, there is an
elimination sequence.

\begin{prop}
  Suppose $\proves{\etl{\LC}} E(e_1, \dots, e_n)$ with a proof in
  which all critical formulas are weak. Then there are terms~$t_i^j$
  such that $\proves{\LC} \bigvee_j E(t_1^j, \dots, t_n^j)$.
\end{prop}
  
\begin{proof}
  Take a critical \et-term~$e$ of maximum degree among those of
  maximal rank in~$\pi$, let $\Gamma(e)$ be the critical formulas
  belonging to~$e$, $\Gamma$ the remaining critical formulas, and
  suppose the end-formula is~$D(e)$. Since all criticial formulas
  $A(u) \lif A(e)$ (or $A(e) \lif A(u)$) are weak, $e$ does not occur
  in~$u$, i.e., all critical formulas in~$\Gamma(e)$ are predicative.
  By Lemma~\ref{lem:pred-lin}, $e$ has an $e$-elimination set~$T$, and
  correspondingly $\Gamma[T] \vdash_{\LC} \bigvee_t D(t)$.
  However, since the critical formulas in~$\Gamma$ are also weak, they
  do not contain~$e$, hence $\Gamma[T] = \Gamma$. The result follows
  by the same inductive proof as the first \et-theorem.
\end{proof}

The extended first \et-theorem guarantees the existence of Herbrand
disjunctions for existential theorems, i.e., if $E \equiv \exists
x_1\dots x_n D(x_1, \dots, x_n)$ and $\vdash E$ then $\vdash \bigvee_{i} D(t_{1i},
\dots, t_{ni})$. The existence of a Herbrand disjunction, conversely,
guarantees the existence of a proof of $E^\et$ for which a predicative
elimination sequence exists.

\begin{prop}\label{prop:disj-pred-elim}
  If $\vdash \bigvee_{i} D(t_{1i},
  \dots, t_{ni})$ then there is an \et-derivation of $[\exists
  x_1\dots\exists x_n\,D(x_1, \dots, x_n)]^\et$ for which a predicative
  elimination sequence exists.
\end{prop}

\begin{proof}
  We give an example only. Suppose
$\proves{\LC} D(s_1, t_1) \lor D(s_2, t_2)$. First, consider
\begin{align*}
  e(x) & \equiv \meps y D(x, y)
  \intertext{Then both}
  C_1(e(s_1)) & \equiv D(s_1, t_1) \lif D(s_1, e(s_1))\\
  C_2(e(s_2)) & \equiv D(s_2, t_1) \lif D(s_2, e(s_2))\\
\intertext{are predicative critical formulas. Since $ \proves{\LC} D(s_1, t_1) \lor D(s_2, t_2)$ we get from them
$D(s_1, e(s_1)) \lor D(s_2, s_2)$. Now let}
e' & \equiv \meps x D(x, e(x)).\\
 C_3(e') & \equiv D(s_1, e(s_1)) \lif D(e', e(e'))\\
 C_4(e') & \equiv D(s_2, e(s_2)) \lif D(e', e(e'))
\intertext{are also predicative critical formulas. Together we have a proof of
$D(e', e(e')) \equiv [\exists x\exists y D(x,y)]^\et$. Then}
& \langle\{C_3(e'), C_4(e')\}, \{s_1, s_2\}\rangle\\
& \langle\{C(_1(e(s_1)))\}, \{t_1\}\rangle\\
& \langle \{C(_2(e(s_2)))\},\{t_2\}\rangle
\end{align*}
is an elimination sequence. In fact it is an elimination sequence
following Hilbert's ordering, since $e'$ has higher rank than
$e(s_1)$ and~$e(s_2)$. In each step, only predicative critical
formulas are generated. It produces the original Herbrand disjunction.
\end{proof}

\section{The First \et-Theorem and Order Induction}
\label{sec:order}

In arithmetic, the usual methods for eliminating critial formulas
based on the extended first \eps-theorem do not work; and so
consistency proofs for systems based on the \eps-calculus use
other methods such as the \eps-substitution method (see
\citealt{Ackermann1940} and \citealt{Moser2006}; the history of the
two approaches is discussed in \citealt{Zach2004a}). The methods
developed for the extended first \et-theorem to eliminate predicative
critical formulas from proofs in $\etl{\LC}$ above can, however, also
be applied in classical theories of order (including arithmetic).
Suppose $T$ is a universal theory involving a relation~$<$, and
consider the order induction rule~\IR,
\[
\AxiomC{$\forall x(x < y \lif A(x)) \lif A(y)$}
\RightLabel{\IR}
\UnaryInfC{$A(t)$}
\DisplayProof
\]
We denote by $\vdash_<$ the derivability relation generated by
classical logic extended by~\IR. The resulting system is equivalent to
adding to classical first-order logic the order induction principle for~$<$,
\begin{equation*}
  \tag{\IP}
  \forall y((\forall x(x < y \lif A(x)) \lif A(y)) \lif \forall z\,A(z).
\end{equation*}

\begin{prop}
  $T \vdash_< A$ iff $T + \IP \vdash A$
\end{prop}

\begin{proof}
  The ``only if'' direction follows by observing that if 
  \begin{align*}
    T & \vdash \forall x(x<y \lif A(x)) \lif A(y) \text{ then also}\\
    T & \vdash \forall y(\forall x(x<y \lif A(x)) \lif A(y))
  \end{align*}
  and so $A(t)$ follows from~\IP{} and $\forall x\,A(x) \lif A(t)$ by
  modus ponens. For the ``if'' direction, let $P_A$ be
  \[\forall y(\forall x(x < y \lif A(x)) \lif A(y))\]
  Then by logic,
  \begin{align*}
    & \vdash \forall u(u < v \lif A(u)) \lif (P_A \lif A(v)) \text{ and so}\\
    & \vdash \forall u(u < v \lif (P_A \lif A(u))) \lif (P_A \lif A(v))\\
    & \vdash_< P_A \lif A(z) \text{ by \IR, and consequently}\\
    & \vdash_< P_A \lif \forall z\,A(z).
  \end{align*}
  Thus, $\vdash_< \IP$.
\end{proof}

Now consider the classical \eps-calculus extended by critical formulas
of the form
\[
  A(t) \lif \lnot t < \meps x A(x)
\] 
These critical formulas are obviously equivalent to Hilbert's
``critical formulas of the second form,'' $A(t) \lif \meps x A(x) \le
t$, over a theory that proves that $<$ is trichotomous. If $A$ is
derivable from $T$ and ordinary critical formulas for \eps-terms and
critical formulas of this second kind, we write $T \vdash_\eo A$.  The
standard translation $A^\eps$ of a formula~$A$ is defined as in
Definition~\ref{defn:et-trans}, except $\forall x\,A(x)^\eps \equiv
A^\eps(\meps x \lnot A^\eps(x))$. Then we can show:

\begin{prop}
  If $T \vdash_< A$ then $T^\eps \vdash_\eo A^\eps$.
\end{prop}

\begin{proof}
  As in the proof of Lemma~\ref{embed}. We just have to deal with
  application of~\IR. Suppose we have a derivation of the
  \eps-translation of the premise of~\IR,
  \begin{align*}
    & (e(y) < y \lif A^\eps(e(y)) \lif A^\eps(y)
  \intertext{where $e(y) \equiv \meps x\lnot (x<y \lif A^\eps(x))$ is the
  \eps-term used in the translation of $\forall x(x < y \lif A(x))$. By 
  substituting $e' \equiv \meps z \lnot A^\eps(z)$ for~$y$ throughout the 
  proof, we obtain}
  & (e(e') < e' \lif A^\eps(e(e')) \lif A^\eps(e')
  \intertext{Take the critical formula of second kind 
  $\lnot A^\eps(t) \lif \lnot t < \meps z \lnot A^\eps(z)$
  and let $t$ be~$e(e')$. By contraposition, we have}
  & e(e') < e' \lif A^\eps(e(e')) \text{ and so}\\
  & A^\eps(e')
  \intertext{by modus ponens. The conclusion of~\IR, $A^\eps(t)$, 
  now follows from an ordinary critical formula
  belonging to $\lnot A^\eps(z)$, viz.,}
  & \lnot A^\eps(t) \lif \lnot A^\eps(e').%\qedhere
\end{align*}  
\end{proof}

Arithmetic does not have a Herbrand theorem, and thus also no first
\eps-theorem. However, Herbrand disjunctions exists for formulas
$\exists \vec x\, E(\vec x)$ iff the critical formulas belonging to
the \eps-terms $e_1$, \dots, $e_n$ in the standard \eps-translation
$E^\eps(e_1, \dots, e_n)$ of $\exists x_1\dots \exists x_n\,E(x_1,
\dots, x_n)$ can be eliminated by a predicative elimination sequence.
This mirrors the situation in \LC{} discussed in
Section~\ref{sec:LIN}. The proof that critical formulas can be
eliminated if a predicative elimination
sequence exists is similar. Corresponding to Lemma~\ref{lem:bigdisj}
we'll need the following:

\begin{lem}\label{lem:disj-less}
  Let $V$ be a finite set of variables and assume that
  \begin{align*}
    T & \vdash \forall x\,\lnot x < x \tag{\textit{Irr}}\label{eq:irr}\\
    T & \vdash \forall x \forall y \forall z((x < y \lif (y < z \lif x < z)) \tag{\textit{Trans}}
    \label{eq:trans}
  \end{align*}
  Then $T \vdash \bigvee_{x \in
  V} \bigwedge_{y \in V} \lnot y < x$.
\end{lem}

\begin{proof}
  Suppose not. Then $T + \bigwedge_{x \in
  V} \bigvee_{y \in V} y < x$ would be satisfiable. Let $\mathfrak{M}$
  and $s$ be the corresponding structure and variable assignment. Fix
  $x_1 \in V$. Since $\mathfrak{M}, s \models \bigvee_{y \in V} y < x_1$,
  for some $x_2 \in V$, $s(x_2) <^\mathfrak{M} s(x_1)$. Continuing, we
  obtain $x_1$, \dots, $x_n \in V$ such that $s(x_{i+1}) <^\mathfrak{M}
  s(x_i)$ for any~$n$. Since $V$ is finite, eventually $x_i \equiv
  x_{i+k}$, contradicting the assumption that any model of $T$ makes
  $<$ irreflexive and transitive.
\end{proof}

\begin{thm}\label{thm:eps-subst}
  Suppose $T$ is as in Lemma~\ref{lem:disj-less}. Then $T \vdash
  \bigvee_{i=1}^k E(t_{1i}, \dots, t_{ni})$ for some terms $t_{ji}$
  iff there is a derivation of $E^\eps(e_1, \dots, e_n)$ which has a
  predicative elimination sequence.
\end{thm}

\begin{proof}
  For the ``only if'' part, proceed as in the procedure outlined in the
  proof of Proposition~\ref{prop:disj-pred-elim}. For the ``if'' part,
  we have to show that if the critical formulas belonging to an
  \eps-term are predicative, they can be eliminated. Without loss of
  generality we may assume that for each term~$t_i$, a corresponding
  critical formula of first and of second kind are both present. So
  suppose $e$ is a critical \eps-term and $\Theta, \Gamma, \Pi(e),
  \Pi'(e) \vdash_\eo D(e)$ where $\Pi(e)$ and and $\Pi'(e)$ consist
  of, respectively,
  \begin{align*}
    A(t_1) & \lif A(e),  & \dots, &  & A(t_m) & \lif A(e) \\
    A(t_1) & \lif \lnot t_1 < e, & \dots,& & A(t_m) & \lif \lnot t_m < e,
  \end{align*}
  and $\Theta$ consists of instances of formulas in~$T$.
  
  Let $W = \{t_1, \dots, t_m\}$. If $V \subseteq W$, let $C_V$ be
  \[
    \bigwedge_{t\in V} A(t) \land \bigwedge_{t \in W\setminus V} \lnot A(t)
  \]
  and let $\Pi_V(e)$ be the critical formulas with terms~$t_i \in V$ and
  $\Pi_{W\setminus V}(e) = \Pi(e) \setminus \Pi_V(e)$, and similarly for
  $\Pi'_V(e)$ and $\Pi'_{T\setminus V}(e)$. Since $\lnot A(t)
  \vdash A(t) \lif B$, we have $C_V \vdash \Pi_{W\setminus V}(e)$ and
  $C_V \vdash \Pi'_{W\setminus V}(e)$, and so
  \begin{align*}
    \Theta, \Gamma, C_V, \Pi_V(e), \Pi'_V(e) & \vdash_\eo D(e)
    \intertext{Since the critical formulas in $\Pi(e)$ are predicative, $t_i$ 
    does not contain~$e$ and so $C_V[t_i/e] = C_V$. Since $C_V \vdash A(t_i)$ 
    for every $t_i \in V$, $C_V \vdash \Pi_V(t_i)$. So we also have}
    \Theta[t_i/e], \Gamma[t_i/e], C_V, \Pi'_V(t_i) & \vdash_\eo D(t_i)
    \intertext{for each $t_i \in V$ and consequently}
    \Theta[V], \Gamma[V], C_V, \bigvee_{t \in V} \bigwedge_{u\in V} 
    (A(u) \lif \lnot u < t) & \vdash_\eo \bigvee_{t \in V} D(t_i)
  \intertext{Since this is true for every $V \subseteq W$, we get}
    \Theta[V], \Gamma[W], \bigvee_{V \subseteq W} C_V, \Xi & 
    \vdash_\eo \bigvee_{t \in W} D(t_i)
  \intertext{if we let $\Xi = \{ B_V:V \subseteq W\}$ where $B_V$ is 
  $\bigvee_{t \in V} \bigwedge_{u\in V} (A(u) \lif \lnot u < t)$. Now 
  $\bigvee_{V \subseteq W} C_V$ is itself provable in classical
  logic. So as the result of one elimination step, we get}
  \Theta[V], \Gamma[W], \Xi & 
  \vdash_\eo \bigvee_{t \in W} D(t_i)
  \intertext{If there is a predicative elimination sequence, we have in 
  the end terms $t_{ij}$ such that}
  \Theta', \Xi' &\vdash_\eo \bigvee_{i=1}^k E(t_{1i}, \dots, t_{ni})
  \intertext{where $\Theta'$ are all the instances of formulas of~$T$ 
  produced in the elimination, and $\Xi'$ are all the formulas of the 
  form~$B_V$, possibly with epsilon terms replaced by other terms. 
  $T \vdash A$ for each $A \in \Theta'$. By Lemma~\ref{lem:disj-less},}
    T & \vdash \bigvee_{t\in V}\bigwedge_{u\in V} \lnot u < t
  \intertext{and so each formula $B_V \in \Xi'$ is also provable from~$T$. 
  Together we have,}
  T &\vdash \bigvee_{i=1}^k E(t_{1i}, \dots, t_{ni})
  \end{align*}
\end{proof}
  
\section{Open Problems}

We have investigated the \et-calculi for intermediate logics, with a
focus on the extended first \et-theorem. We showed that the only
intermediate logics with an  extended first \et-theorem are the
finite-valued G\"odel logics~$\LC_m$, but obtained partial results for
formulas of specific form or with specific kinds of proofs for other
intermediate logics.  

The natural next question to investigate is the second \eps-theorem,
i.e., to investigate extended \et-calculi for intermediate predicate
logics and characterize those logics for which the extended
\et-calculus is conservative. We have shown that $\etx{\Lo}$ proves
all quantifier shifts, so $\etx{\Lo}$ is not conservative over any
$\ql{\Lo}$ where these are not provable. Note that this question is
not automatically settled by the answer to the question of which
logics have the extended first \et-theorem. Rather, it is a question orthogonal
and requires other proof systems for a proper investigation, such as
sequent calculi for \et-terms. For instance, \citet[Theorem
5.4]{AguileraBaaz2019} show that the standard translation~$\Gamma
\Rightarrow D^\eps$ of a sequent~$\Gamma \Rightarrow D$ is provable in
a sequent calculus $\mathbf{LJ}^\eps$ for intuitionistic logic iff
$\Gamma \Rightarrow D$ is provable in a special version
$\mathbf{LJ}^{++}$ of intuitionistic sequent calculus which is
globally sound but allows violation of the eigenvariable condition.
$\mathbf{LJ}^{++}$ in turn proves $\Gamma \Rightarrow D$ iff
$\mathbf{LJ} + \CD + \Aout + \Eout \vdash \Gamma \Rightarrow D$
(Proposition~4.4). In other words, the second \eps-theorem holds for
intutionistic predicate logic with all quantifier shifts. 

The methods used here are closely related to the study of the behavior
of Skolem functions in intermediate logics (of which \et-terms are in
many ways a syntactic variant), see, e.g., \cite{Iemhoff2019}. As
mentioned in the introduction, other approaches to adding
\eps-operators to intuitionistic logic yield systems that are
conservative over the original logic. Work on Skolemization in
intuitionistic logic is relevant here, and suggests that conservative
\eps-calculi can be obtained by introducing existence predicates. The
proof-theoretic approaches in the literature would benefit also from a
complementary model-theoretic study. A Kripke-style semantics for
\et-terms, with or without existence predicate, is still lacking (but
see \citealt{DeVidi1995} for a semantics based on Heyting algebras).

In Section~\ref{sec:LIN} we gave sufficient conditions for when
\et-terms can be eliminated from a proof~$\pi$ in $\etl{\LC}$. Are
there better (weaker) criteria that apply to more proofs? For
instance, there may be certain kinds of orderings such that if the
critical \et-terms in $\pi$ and corresponding ``witness terms'' can be
put into such an ordering, an elimination sequence in which only
predicative critical formulas exists. The same question applies for
the parallel condition in arithmetical theories in
Theorem~\ref{thm:eps-subst}.

\subsection*{Acknowledgements} The authors would like to thank Guram
Bezhanishvili, David Gabelaia, and the reviewer for the \textsc{Journal}.

\bibliographystyle{asl}
%\bibliography{epsncl}

\end{document}